\documentclass[a4paper,12pt,reqno]{amsart}

\usepackage{amssymb,amsmath,array,amscd,hhline}

\usepackage[mathscr]{euscript}
\usepackage{stmaryrd}
\usepackage{ulem}
\usepackage{xr}

\usepackage{mathrsfs}

\def\SSigma_#1{\BS(#1)}

\usepackage{tikz-cd}

\def\kappa{\varkappa}

\usepackage{mathrsfs}

\def\({\left(}
\def\){\right)}
\def\[{\left[}
\def\]{\right]}

\def\GL{\mathop{\rm GL}\nolimits}
\def\Gr{\mathop{\rm Gr}\nolimits}

\def\k{\Bbbk}

\def\Frac{\mathop{\rm Frac}\nolimits}

\def\ind{\mathop{\rm ind}\nolimits}

\def\H{\mathop{\mathcal H}\nolimits}

\def\dlim{\mathop{\rm lim}\limits_{\longrightarrow}}

\def\lm{\lambda}

\def\C{\mathbb C}

\def\Z{\mathbb Z}

\def\F{\mathscr F}

\def\BS{\mathrm{BS}}

\def\G{\mathscr G}

\def\<{\langle}
\def\>{\rangle}

\def\n{\mathfrak{n}}
\def\r{\mathfrak{r}}

\renewcommand\emptyset{\varnothing}
\renewcommand\phi{\varphi}

\def\im{\mathop{\rm im}}
\def\id{\mathrm{id}}
\renewcommand\epsilon{\varepsilon}

\def\X{\mathcal X}
\def\suchthat{\mathbin{\rm |}}

\def\ito{\stackrel\sim\to}
\def\pt{{\rm pt}}
\def\pr{\mathop{\rm pr}\nolimits}

\def\q{\mathfrak q}

\def\SL{\mathop{\rm SL}\nolimits}
\def\SU{\mathop{\rm SU}\nolimits}

\def\Ob{\mathop{\rm Ob}}
\def\Lotimes{\stackrel{L}\otimes}

\newtheorem{theorem}{Theorem}
\newtheorem{proposition}[theorem]{Proposition}
\newtheorem{lemma}[theorem]{Lemma}

\newtheorem{corollary}[theorem]{Corollary}

\newtheorem{definition}[theorem]{Definition}
\def\gld{\mathop{\rm gld}}
\def\span{\mathop{\rm span}}

\def\le{\leqslant}
\def\ge{\geqslant}

\def\csh#1#2{\underline{#1}{}_{{}_{\scriptstyle #2}}}

\renewcommand{\labelenumi}{{\rm\theenumi}}
\renewcommand{\theenumi}{{\rm(\arabic{enumi})}}

\def\={\equiv}

\def\2{{(2)}}

\title[Decompositions for cohomologies of Bott-Samelson varieties]{Tensor product decompositions for cohomologies of Bott-Samelson varieties}
\author{Vladimir Shchigolev}
\address{Financial University under the Government of the Russian Federation\\49 Leningradsky Prospekt, Moscow, Russia}
\email{shchigolev\_vladimir@yahoo.com}

\begin{document}

\maketitle

\begin{abstract}
Let $T$ be a maximal torus of a semisimple complex algebraic group, $\BS(s)$ be
the Bott-Samelson variety for a sequence of simple reflections $s$ and $\BS(s)^T$
be the set of $T$-fixed points of $\BS(s)$. We prove the tensor product decompositions
for the image of the restriction $H^\bullet_T(\BS(s),\k)\to H_T^\bullet(X,\k)$,
where $X\subset\BS(s)^T$ is defined by some special not overlapping
equations $\gamma_i\gamma_{i+1}\cdots\gamma_j=w_{i,j}$ with right-hand sides
belonging to the Weyl group.
\end{abstract}

\bigskip

\section{Introduction}

Let $G$ be a semisimple complex algebraic group. We fix a Borel subgroup $B\le G$ and a maximal torus $T\le B$
and denote by $W$ the corresponding Weyl group.
For a sequence of simple reflections $s=(s_1,\ldots,s_n)$, we consider the {\it Bott-Samelson} variety
$$
\BS(s)=P_{s_1}\times\cdots\times P_{s_n}/B^n
$$
where $P_{s_i}=B\cup Bs_iB$ is the minimal parabolic subgroup.
The torus $T$ acts on $\BS(s)$ via the first component. So we can consider
the set of $T$-fixed points $\BS(s)^T$, which can be identified with the set of
combinatorial galleries
$$
\Gamma(s)=\{(\gamma_1,\ldots,\gamma_n)\in W^n\suchthat\forall
i:\gamma_i=1\text{ or }\gamma_i=s_i\}
$$
Let $\k$ be a principal ideal domain with invertible $2$. Then any restriction
$H_T^\bullet(\BS(s),\k)\to H_T^\bullet(\Gamma(s),\k)$ is injective.

Suppose that $X$ is a subset of $\Gamma(s)$. We would like to know what is the image of
the restriction
\begin{equation}\label{eq:intr:1}
H_T^\bullet(\BS(s),\k)\to H_T^\bullet(X,\k).
\end{equation}
This question stated in this generality does not have an answer. However, the answer is know in the following two cases:
$X=\Gamma(s)$ and $X=\{(\gamma_1,\ldots,\gamma_n)\in\Gamma(s)\suchthat\gamma_1\cdots\gamma_n=w\}$ for some $w\in W$.
We denote the image of~(\ref{eq:intr:1}) by $\X(s)$ in the first case and by $\X(s,w)$ in the second case.
Note that $\X(s)\cong H_T^\bullet(\BS(s),\k)$ and $\X(s,w)\cong H_T^\bullet(\pi(s)^{-1}(wB),\k)$, where
$\pi(s):\BS(s)\to G/B$ is the map $(g_1,\ldots,g_n)B^n\mapsto g_1\cdots g_nB$.
The description of $\X(s)$ and $\X(s,w)$ by congruences is given by M.H\"arterich~\cite{Haerterich}.
See also~\cite{BTECBSV} for the coefficients different from fields of characteristic zero.

Instead of a single restriction $\gamma_1\cdots\gamma_n=w$ in the definition of $X$, we can try to impose multiple
restrictions of this form: we take some indexing set of pairs
$
R\subset\{(r_1,r_2)\subset\Z^2\suchthat1\le r_1\le r_2\le n\}
$
and a map $v:R\to W$ and consider the set
$$
X=\{(\gamma_1,\ldots,\gamma_n)\in\Gamma(s)\suchthat\forall r\in R\;\gamma_{r_1}\cdots\gamma_{r_2}=v_r\}.
$$
Let $\X(s,v)$ denote the image of the restriction $H_T^\bullet(\BS(s),\k)\to H_T^\bullet(X,\k)$.
In this paper, we consider only the case where the imposed restrictions {\it do not overlap}, that is,
the case $r_1\le r'_1\le r_2\le r'_2$ is impossible for distinct $r,r'\in R$.

The main results
of this paper
are Theorems~\ref{theorem:bt7:1} and~\ref{theorem:bt7:2}.
They assert that under some restrictions on $(s,v)$,
the module $\X(s,v)$
can be decomposed into a tensor product of some modules $\X(t)$ and $\X(t,w)$ for the corresponding
sequences $t$ and elements $w\in W$.
Moreover, this decomposition is given by the exact formula, see~(\ref{eq:bt7:3}).
The restrictions imposed on $(s,v)$ claim that certain sequences of reflections
associated to this pair be of gallery type, see Definitions~\ref{definition:1} and~\ref{definition:3}.

Perhaps the main application of this result is the possibility to construct elements
of the dual module of $\X(s,v)$ from the elements of the dual modules of the factors $\X(t)$
and $\X(t,w)$ entering into the tensor product decomposition, see Corollaries~\ref{corollary:1} and~\ref{corollary:2}.

The paper is organized as follows. In Section~\ref{Bott-Samelson varieties}, we
fix the notation for simisimple complex algebraic groups, their compact subgroups,
Bott-Samelson varieties $\BS(s)$ and their compactly defined variants $\BS_c(s)$.
To avoid renumeration, we consider the sequences as maps defined on totaly ordered sets,
see Section~\ref{Sequences and groups}.

In Section~\ref{Nested_fibre_bundles}, we recall the main constructions from~\cite{NFBBSV}
concerning nested structures, the subspaces $\BS_c(s,v)$ of $\BS_c(s)$ and the corresponding fibre bundles.
In Section~\ref{Def_equiv}, we consider equivariant cohomologies. The main technical tool here
is the Stiefel manifolds, see Section~\ref{Stiefel manifolds} for the definitions.
In Section~\ref{Twisted action of S}, we also consider the {\it twisted} actions of the equivariant cohomology
of the point on $H^\bullet_K(\BS_c(s),\k)$, where $K$ is a maximal compact torus.
These actions become important later in Section~\ref{The_main_result},
where we consider tensor products of bimodules.

Section~\ref{Tensor products} is devoted to
the proof of Theorem~\ref{theorem:bt7:3} about the tensor product decomposition for $K$-equivariant cohomologies
of the spaces $\BS_c(s,v)$. Here we use the same approach as in~\cite{NFBBSV}, that is,
first we embed Borel constructions into products of Borel constructions (Lemma~\ref{lemma:10}) and then
compute the cohomology of the difference to assure the surjectivity of restriction (Lemma~\ref{lemma:13}).
In this case, we also use fibre bundles for differences (Lemma~\ref{lemma:11}),
as we did in~\cite[Lemma~22]{NFBBSV}.
However in the present case, we need to consider the additional fibre bundle for the composition,
see part~\ref{lemma:11:p:2} of Lemma~\ref{lemma:11}.

Finally, in Section~\ref{Results for the big torus} we come back to the $T$-equivariant cohomologies of $\BS(s)$.
We consider the images $\X(s,v)$ of restriction~(\ref{eq:intr:1}). Theorems~\ref{theorem:bt7:1} and~\ref{theorem:bt7:2}
provide their decompositions into tensor products.
Based on these results, we construct a map for dual modules
in Section~\ref{Dual modules}.

In Section~\ref{Example of a decomposition}, we give an example. It actually states that
$$
H^\bullet_T(\pi(t)^{-1}(\omega_3\omega_4\omega_3B),\k)\otimes_S
H^\bullet_T(\pi(r)^{-1}(\omega_1B),\k)\cong H^\bullet_T(\pi(s)^{-1}(\omega_1\omega_3\omega_4\omega_3B),\k),
$$
where $\omega_1,\omega_2,\omega_2,\omega_4$ are the simple roots for $G=\GL_5(\C)$,
$t=(\omega_4,\omega_3,\omega_4,\omega_3,\omega_4)$, $r=(\omega_2,\omega_1,\omega_2,\omega_1,\omega_2)$
and $s=(\omega_4, \omega_3,  \omega_4, \omega_2,\omega_1,\omega_2,\omega_1,\omega_2, \omega_3,\omega_4)$.
The minimal degree element of the dual module for the right hand-side is a nontrivially twisted product
of the minimal degree elements of the dual modules for the cohomologies of the left-hand side.

As in our previous paper~\cite{NFBBSV}, we consider here only the sheaf cohomology.
We use some standard notation like $\csh{\k}{X}$ for the constant sheaf on a topological space $X$,
$\GL(V)$ for the group of linear isomorphisms $V\to V$, $M_{m,n}(\k)$ for the set of $m\times n$-matrices
with entries in $\k$ and $U(n)$ for the group of unitary $n\times n$-matrices. We also apply
spectral sequences associated with fibre bundles. All of them are first quadrant sequences.

Also note that although principal ideal domains with invertible $2$ are
our main rings of coefficients, we impose on them as little restrictions
as we can in order to prove each separate result about cohomologies.

\section{Bott-Samelson varieties}\label{Bott-Samelson varieties}

\subsection{Compact subgroups of complex algebraic groups}\label{Compact subgroups of complex algebraic groups}
Let $G$ be a semisimple complex group with root system $\Phi$. We assume that $\Phi$ is defined with respect
to the Euclidian space $E$ with the scalar product $(\cdot,\cdot)$. This scalar product
determines a metric on $E$ and therefore a topology. We will denote
by $\overline{A}$ the closure of a subset $A\subset E$.

We consider $G$ as a Chevalley group generated by the root
elements $x_{\alpha}(t)$, where $\alpha\in\Phi$ and $t\in\C$
(see~\cite{Steinberg}). We also fix the following elements of $G$:
$$
w_\alpha(t)=x_\alpha(t)x_{-\alpha}(-t^{-1})x_\alpha(t),\quad
\omega_\alpha=w_\alpha(1),\quad
h_\alpha(t)=w_\alpha(t)\omega_\alpha^{-1}.
$$

For any $\alpha\in\Phi$, we denote by $s_\alpha$ the reflection of $E$
through the hyperplane $L_\alpha$ of the vectors perpendicular to
$\alpha$. We call $L_\alpha$ a {\it wall} perpendicular to
$\alpha$. These reflections generate the Weyl group $W$. We choose
a decomposition $\Phi=\Phi^+\sqcup\Phi^-$ into positive and
negative roots and write $\alpha>0$ (resp. $\alpha<0$) if
$\alpha\in\Phi^+$ (resp. $\alpha\in\Phi^-$). Let
$\Pi\subset\Phi^+$ be the set of simple roots corresponding to
this decomposition. We denote
$$
\mathcal T(W)=\{s_\alpha\suchthat\alpha\in\Phi\},\quad \mathcal
S(W)=\{s_\alpha\suchthat\alpha\in\Pi\}.
$$
and call these sets the {\it set of reflections} and the {\it set
of simple reflections} respectively. We use the standard notation
$\<\alpha,\beta\>=2(\alpha,\beta)/(\beta,\beta)$.

There is the analytic automorphism $\sigma$ of $G$ such that
$\sigma(x_\alpha(t))=x_{-\alpha}(-\bar t)$~\cite[Theorem 16]{Steinberg}.
Here and in what follows $\bar t$ denotes the
complex conjugate of $t$. We denote by $C=G_\sigma$ the set of
fixed points of this automorphism. We also consider the compact
torus $K=T_\sigma=B_\sigma$, where $T$ is the subgroup of $G$
generated by all $h_\alpha(t)$ and $B$ is the subgroup generated
by $T$ and all root elements $x_\alpha(t)$ with $\alpha>0$.

For $\alpha\in\Phi$, we denote by $G_\alpha$ the subgroup of $G$
generated by all root elements $x_\alpha(t)$ and $x_{-\alpha}(t)$
with $t\in\C$. There exists the homomorphism
$\phi_\alpha:\SL_2(\C)\to G$ such that
$$
\(\!
\begin{array}{cc}
1&t\\
0&1
\end{array}\!
\) \mapsto x_\alpha(t),\quad \(\!
\begin{array}{cc}
1&0\\
t&1
\end{array}\!
\) \mapsto x_{-\alpha}(t).
$$
Hence, we get
$$
C\cap G_\alpha=\phi_\alpha(\SU_2).
$$

Let $\mathcal N$ be the subgroup of $G$ generated by the elements
$\omega_\alpha$ and the torus $K$. Clearly $\mathcal N\subset C$.
The arguments of~\cite[Lemma~22]{Steinberg} prove that there
exists an isomorphism $\phi:W\stackrel\sim\to\mathcal N/K$ given
by $s_\alpha\mapsto \omega_\alpha K$.
We choose once and for all, a representative $\dot w\in\phi(w)$ for any $w\in W$. 
Abusing notation, we denote $wK=\dot wK=\phi(w)$.


Now consider the product $C_\alpha=\phi_\alpha(\SU_2)K$.
This is a compact and closed subgroup of $G$.
Moreover, $C_\alpha=C_{-\alpha}$. So we can denote $C_{s_\alpha}=C_\alpha$.
For any $w\in W$, we have $\dot wC_{s_\alpha}\dot w^{-1}=C_{ws_\alpha w^{-1}}=C_{w\alpha}$.

\subsection{Sequences and products}\label{Sequences and groups}
Let $I$ be a finite totally ordered set. We denote by $\le$ ($<$, $\ge$, etc.) the order on it.
We also add two additional elements $-\infty$ and $+\infty$ with
the natural properties: $-\infty<i$ and $i<+\infty$ for any $i\in
I$ and $-\infty<+\infty$. For any $i\in I$, we denote by $i+1$
(resp. $i-1$) the next (resp. the previous) element of
$I\cup\{-\infty,+\infty\}$. We write $\min I$ and $\max I$ for the
minimal and the maximal elements of $I$ respectively and denote
$I'=I\setminus\{\max I\}$ if $I\ne\emptyset$. Note that
$\min\emptyset=-\infty$ and $\max\emptyset=+\infty$. For $i,j\in
I$, we set $[i,j]=\{k\in I\suchthat i\le k\le j\}$.

Any map $s$ from $I$ to an arbitrary set is called a {\it
sequence} on $I$. We denote by $|s|$ the cardinality $|I|$
of $I$ (we will use this notation for any finite set) and call this number the {\it length} of $s$.
For any $i\in I$, we denote by $s_i$ the value of $s$ on $i$.
Finite sequences in the usual sense are sequences on the initial intervals
of the natural numbers, that is, on the sets $\{1,2,\ldots,n\}$.
For any sequence $s$ on a nonempty set $I$, we denote by $s'=s|_{I'}$ its truncation.

We often use the following notation for Cartesian products:
$$
\prod_{i=1}^nX_i=X_1\times X_2\times\cdots\times X_n
$$
and denote by $p_{i_1i_2\cdots i_m}$ the projection of this product to
$X_{i_1}\times X_{i_2}\times\cdots\times X_{i_m}$.
Similar notation is used for tensor products over a commutative ring $\k$:
$$
\mathop{{\bigotimes}}\limits_{i\in I}{}_\k\,M_i=M_{i_1}\otimes_\k\cdots\otimes_\k M_{i_n},
$$
where $I=\{i_1,\ldots,i_n\}$, the elements $i_1,\ldots,i_n$ are distinct and $M_{i_1},\ldots,M_{i_n}$
are $\k$-modules either left of right, which is not important as we consider all such modules automatically
as $\k$-$\k$-bimodules.

If we have maps $f_i:Y\to X_i$ for all
$i=1,\ldots,n$, then we denote by $f_1\boxtimes\cdots\boxtimes f_n$ the map that takes
$y\in Y$ to the $n$-tuple $(f_1(y),\ldots,f_n(y))$. To avoid too many superscripts, we denote
$$
X(I)=\prod_{i\in I}X.
$$
We consider all such products with respect to the product topology if all factors are topological spaces.
Suppose that $X$ is a group. Then $X(I)$ is a topological group with respect to the componentwise
multiplication. In that case, for any $i\in I$ and $x\in X$, we
consider the {\it indicator sequence} $\delta_i(x)\in X(I)$ defined by
$$
\delta_i(x)_j=\left\{
\begin{array}{ll}
x&\text{ if }j=i;\\
1&\text{ otherwise}
\end{array}
\right.
$$
Let us additionally assume that $I$ is embedded into a totally ordered set $J$.
Then for any 
$j\in J\cup\{-\infty\}$
and a sequence $\gamma$ on $I$, we will use the notation
$\gamma^j=\gamma_{\min I}\gamma_{\min I+1}\cdots\gamma_i$,
where $i$ is the maximal element of $I$ less than or equal to $j$ or $-\infty$ if there is no such element.
Obviously $\gamma^{-\infty}=1$. We set
$\gamma^{\max}=\gamma^{\max I}$.

We will also use the following notation. Let $G$ be a group. If $G$ acts on the set $X$, then we denote by $X/G$
the set of $G$-orbits of elements of $X$. Suppose that $\phi:X\to Y$ is a $G$-equivariant map.
Then $\phi$ maps any $G$-orbit to a $G$-orbit. We denote the corresponding map from $X/G$ to $Y/G$
by $\phi/G$.

\subsection{Galleries}\label{Galleries}
As $L_{-\alpha}=L_\alpha$ for any root $\alpha\in\Phi$, we can denote $L_{s_\alpha}=L_\alpha$.
For any $w\in W$, we get
$wL_{s_\alpha}=wL_\alpha=L_{w\alpha}=L_{s_{w\alpha}}=L_{ws_\alpha
w^{-1}}$.

A {\it chamber} is a connected component of the space
$E\setminus\bigcup_{\alpha\in\Phi}L_\alpha$. We denote by
$\mathbf{Ch}$ the set of chambers and by $\mathbf L$ the set of
walls $L_{s_\alpha}$. We say that a chamber $\Delta$ is {\it
attached} to a wall $L$ if the intersection $\overline\Delta\cap
L$ has dimension $\dim E-1$. The {\it fundamental} chamber is
defined by
$$
\Delta_+=\{v\in E\suchthat (v,\alpha)>0\text{ for any
}\alpha\in\Pi\}.
$$
A {\it labelled gallery} on a finite totaly ordered set $I$ is a
pair of maps $(\Delta,\mathcal L)$, where
$\Delta:I\cup\{-\infty\}\to\mathbf{Ch}$ and $\mathcal
L:I\to\mathbf L$ are such that for any $i\in I$,
both chambers $\Delta_{i-1}$ and $\Delta_i$ are attached to
$\mathcal L_i$. The Weyl group $W$ acts on the set of labelled
galleries by the rule $w(\Delta,\mathcal
L)=(w\Delta,w\mathcal L)$, where $(w\Delta)_i=w\Delta_i$ and
$(w\mathcal L)_i=w\mathcal L_i$.

For a sequence $s:I\to\mathcal T(W)$, we define the set of {\it
generalized combinatorial galleries}
$$
\Gamma(s)=\{\gamma:I\to W\suchthat\gamma_i=s_i\text{ or
}\gamma_i=1\text{ for any }i\in I\}.
$$
If $s$ maps $I$ to $\mathcal S(W)$, then the elements of
$\Gamma(s)$ are called {\it combinatorial galleries}.
To any combinatorial galley $\gamma\in\Gamma(s)$ there corresponds
the following labelled gallery:
\begin{equation}\label{eq:5}
i\mapsto \gamma^i\Delta_+,\qquad i\mapsto
L_{\gamma^is_i(\gamma^i)^{-1}}.
\end{equation}
%
Moreover, any labelled gallery $(\Delta,\mathcal L)$ on $I$ such
that $\Delta_{-\infty}=\Delta_+$ can be obtained this way (for the
corresponding $s$ and $\gamma$).

Let $\gamma\in\Gamma(s)$, where $s:I\to\mathcal T(W)$.
%
We consider the sequence $s^{(\gamma)}:I\to\mathcal T(W)$ defined
by
$$
s^{(\gamma)}_i=\gamma^is_i(\gamma^i)^{-1}.
$$
To understand, what happens if we apply this operation several
times, we consider the {\it composition} of generalized
combinatorial galleries. Let $\lm\in\Gamma(s^{(\gamma)})$ be
another generalized combinatorial gallery. Then we define
$\gamma\circ\lm\in\Gamma(s)$ by
$$
(\gamma\circ\lm)_i=(\gamma^{i-1})^{-1}\lm_i\gamma^i.
$$
It is easy to check that $(\gamma\circ\lm)^i=\lm^i\gamma^i$. Hence
$$
\big(s^{(\gamma)}\big)^{(\lm)}=s^{(\gamma\circ\lm)}.
$$
Moreover, this composition is associative: for any
$\delta\in\Gamma(s^{(\gamma\circ\lm)})$, we have
$$
(\gamma\circ\lambda)\circ\delta=\gamma\circ(\lambda\circ\delta).
$$
We also get $\gamma\circ\epsilon=\gamma$ and
$\epsilon\circ\lm=\lm$, where $\epsilon$ is the map $I\to\{1\}$.
The inverse of $\gamma$ is $\gamma^{-1}\in\Gamma(s^{(\gamma)})$
given by
$$
(\gamma^{-1})_i=\gamma^{i-1}(\gamma^i)^{-1}.
$$
We have $(\gamma^{-1})^i=(\gamma^i)^{-1}$. Applying this fact, it
is easy to prove that
$\gamma\circ\gamma^{-1}=\gamma^{-1}\circ\gamma=\epsilon$ and
$(\gamma^{-1})^{-1}=\gamma$. Finally, note that
$s^{(\epsilon)}=s$.

Let $w\in W$ and $s:I\to\mathcal T(W)$. Then we define the
sequence $s^w:I\to\mathcal T(W)$ by $(s^w)_i=ws_iw^{-1}$. For any
$\gamma\in\Gamma(s)$, we denote by $\gamma^w$ the generalized
combinatorial gallery of $\Gamma(s^w)$
defined by $\gamma_i^w=w\gamma_i w^{-1}$.

\begin{definition}\label{definition:1}
A sequence $s:I\to\mathcal T(W)$ is of galley type if there exists
a labelled gallery $(\Delta,\mathcal L)$ on $I$ such that
$\mathcal L_i=L_{s_i}$ for any $i\in I$.
\end{definition}

In this case, there exist an element $x\in W$, a sequence of simple
reflections $t:I\to\mathcal S(W)$ and a combinatorial gallery
$\gamma\in\Gamma(t)$ and  such that $x\Delta_{-\infty}=\Delta_+$
and $x(\Delta,\mathcal L)$ corresponds to $\gamma$
by~(\ref{eq:5}).
We call $(x,t,\gamma)$ a {\it gallerification} of $s$.
This fact is equivalent to $t^{(\gamma)}=s^x$.

Note that for any $w\in W$ a sequence $s$ is of gallery type if and only if the sequence $s^w$
is of gallery type.

{\bf Remark.} If $\Phi$ is of type $A_1$ or $A_2$, then every sequence $s:I\to\mathcal T(W)$
is of gallery type~\cite[Example]{NFBBSV}.

\subsection{Definition via the Borel subgroup} For any simple reflection $t\in\mathcal S(W)$,
we consider the minimal parabolic subgroup $P_t=B\cup BtB$. Note
that $C_t$ is a maximal compact subgroups of $P_t$ and
$C_t=P_t\cap C$. Let $s:I\to\mathcal S(W)$ be a sequence of simple
reflections. We consider the space
$$
P(s)=\prod_{i\in I}P_{s_i}
$$
with respect to the product topology. The group $B(I)$ acts on
$P(s)$ on the right by $(pb)_i=b_{i-1}^{-1}p_ib_i$. Here and in what follows, we assume $b_{-\infty}=1$. Let
$$
\BS(s)=P(s)/B(I).
$$
This space is compact and Hausdorff. The Borel subgroup $B$ acts continuously on $P(s)$ by
$$
(bp)_i=\left\{
\begin{array}{ll}
bp_i&\text{ if }i=\min I;\\
p_i&\text{ otherwise.}
\end{array}
\right.
$$
As this action commutes with the right action of $B(I)$ on $P(s)$
described above, we get the left action of $B$ on $\BS(s)$.

We also have the map $\pi(s):\BS(s)\to G/B$ that maps $pB(I)$ to $p^{\max}B$.
This map is invariant under the left action of $B$. However, we
need here only the action of the torus $T$. We have
$\BS(s)^T=\{\gamma B(I)\suchthat\gamma\in\Gamma(s)\}$, where
$\gamma$ is identified with the sequence $i\mapsto\dot\gamma_i$ of
$P(s)$. We will identify $\BS(s)^T$ with $\Gamma(s)$. For any
$w\in\Gamma(s)$, we set $\BS(s,w)=\pi(s)^{-1}(wB)$.

\subsection{Definition via the compact torus}\label{Definition via the compact torus}

Let $s:I\to\mathcal T(W)$ be a sequence of reflections. We
consider the space
$$
C(s)=\prod_{i\in I}C_{s_i}
$$
with respect to the product topology. The group $K(I)$ acts on
$C(s)$ on the right by $(ck)_i=k_{i-1}^{-1}c_ik_i$. Here and in what follows, we assume $k_{-\infty}=1$. Let
$$
\BS_c(s)=C(s)/K(I).
$$
We denote by $[c]$ the right orbit $cK(I)$ of $c\in C(s)$. 
The space $\BS_c(s)$ is compact and Hausdorff.

The action of $K(I)$ on $C(s)$ described above commutes with the
following left action of $K$:
$$
(kc)_i=\left\{
\begin{array}{ll}
kc_i&\text{ if }i=\min I;\\
c_i&\text{ otherwise.}
\end{array}
\right.
$$
Therefore, $K$ acts continuously on $\BS_c(s)$ on the left: $k[c]=[kc]$.

We also have the map $\pi_c(s):\BS_c(s)\to C/K$ that maps
$[c]$ to $c^{\max}K$. This map is obviously invariant under the left action of $K$. We
have $\BS_c(s)^K=\{[\gamma]\suchthat\gamma\in\Gamma(s)\}$, where
$\gamma$ is identified with the sequence $i\mapsto\dot\gamma_i$ of
$C(s)$. We also identify $\BS_c(s)^K$ with $\Gamma(s)$. For any
$w\in W$, we set $\BS_c(s,w)=\pi_c(s)^{-1}(wK)$.
The set of $K$-fixed points of this space is
$$
\BS_c(s,w)^K=\{\gamma\in\Gamma(s)\suchthat\gamma^{\max}=w\}=\Gamma(s,w).
$$

For any $w\in W$, let $l_w:C/K\to C/K$ and $r_w:C/K\to C/K$ be the
maps given by $l_w(cK)=\dot w cK$ and $r_w(cK)=c\dot wK$. They are
obviously well-defined and $l_w^{-1}=l_{w^{-1}}$,
$r_w^{-1}=r_{w^{-1}}$. Moreover, $l_w$ and $r_{w'}$ commute for
any $w,w'\in W$ .
%
%
Let $d_w:\BS_c(s)\to\BS_c(s^w)$ be the map given by
$d_w([c])=[c^w]$, where $(c^w)_i=\dot wc_i\dot w^{-1}$. This map
is well-defined and a homeomorphism. Moreover, $d_w(ka)=\dot
wk\dot w^{-1}a$ for any $k\in K$ and $a\in\BS_c(s)$. We have the
commutative diagram
$$
\begin{tikzcd}[column sep=15ex]
\BS_c(s)\arrow{r}{d_w}[swap]{\sim}\arrow{d}[swap]{\pi_c(s)}&\BS_c(s^w)\arrow{d}{\pi_c(s^w)}\\
C/K\arrow{r}{l_wr_w^{-1}}[swap]{\sim}&C/K
\end{tikzcd}
$$
Hence for any $x\in W$, the map $d_w$ yields the isomorphism
\begin{equation}\label{eq:6}
\BS_c(s,x)\cong\BS(s^w,wxw^{-1}).
\end{equation}

There is a similar construction for a generalized combinatorial gallery
$\gamma\in\Gamma(s)$.
We define the map $D_\gamma:\BS_c(s)\to\BS_c(s^{(\gamma)})$ by
$$
D_\gamma([c])=[c^{(\gamma)}],\text{ where
}c^{(\gamma)}_i=(\dot\gamma_{i-1}\dot\gamma_{i-2}\cdots\dot\gamma_{\min
I})^{-1}c_i\dot\gamma_i\dot\gamma_{i-1}\cdots\dot\gamma_{\min I}.
$$
This map is well-defined. Indeed let $[c]=[\tilde c]$ for some
$c,\tilde c\in C(s)$. Then $\tilde c=ck$ for some $k\in K(I)$.
Then $\tilde c^{(\gamma)}=c^{(\gamma)}k'$, where
$k'_i=(\dot\gamma_i\dot\gamma_{i-1}\cdots\dot\gamma_{\min I})^{-1}k_i\dot\gamma_i\dot\gamma_{i-1}\cdots\dot\gamma_{\min I}$.
The map $D_\gamma$ is obviously continuous and a $K$-equivariant
homeomorphism. We have the following commutative diagram:
$$
\begin{tikzcd}[column sep=15ex]
\BS_c(s)\arrow{r}{D_\gamma}[swap]{\sim}\arrow{d}[swap]{\pi_c(s)}&\BS_c(s^{(\gamma)})\arrow{d}{\pi_c(s^{(\gamma)})}\\
C/K\arrow{r}{r_{\gamma^{\max I}}^{-1}}[swap]{\sim}&C/K
\end{tikzcd}
$$
Hence for any $x\in W$, the map $D_\gamma$ yields the isomorphism
\begin{equation}\label{eq:7}
\BS_c(s,x)\cong\BS_c(s^{(\gamma)},x(\gamma^{\max})^{-1}).
\end{equation}

\subsection{Isomorphism of two constructions}\label{Isomorphism_of_two_constructions}
Consider the Iwasawa decomposition $G=CB$
(see, for example,~\cite[Theorem 16]{Steinberg}). Hence for any
simple reflection $t$, the natural embedding $C_t\subset P_t$
induces the isomorphism of the left cosets $C_t/K\stackrel\sim\to P_t/B$. 
Therefore, for any sequence of simple reflections $s:I\to\mathcal
S(W)$, the natural embedding $C(s)\subset P(s)$ induces the
homeomorphism $\BS_c(s)\stackrel\sim\to\BS(s)$. This homeomorphism
is clearly $K$-equivariant. Moreover, the following diagram is
commutative:
$$
\begin{tikzcd}
\BS_c(s)\arrow{r}{\sim}\arrow{d}&\BS(s)\arrow{d}\\
C/K\arrow{r}{\sim}&G/B
\end{tikzcd}
$$
Note that $\BS_c(s)^K$ is mapped to $\BS(s)^K=\BS(s)^T$. These
sets are identified with $\Gamma(s)$.
%

Any isomorphism of finite totally ordered set $\iota:J\to I$
induces the obvious homeomorphisms
$\BS(s)\stackrel\sim\to\BS(s\iota)$ and
$\BS_c(s)\stackrel\sim\to\BS_c(s\iota)$. They obviously respect
the isomorphisms described above.

\section{Nested fibre bundles}\label{Nested_fibre_bundles}

\subsection{Nested structures}\label{Nested_structures} Here we recall the main constructions of~\cite{NFBBSV}.
Let $I$ be a finite totally ordered set.
We denote the order on it by $\le$ (respectively, $\ge$, $<$
etc.).
Let $R$ be a subset of $I^2$. For any $r\in R$, we denote by $r_1$
and $r_2$ the first and the second component of $r$ respectively.
So we have $r=(r_1,r_2)$. We will also use the notation
$[r]=[r_1,r_2]$ for the intervals. We say that $R$ is a {\it
nested structure} on $I$ if
\smallskip
\begin{itemize}
\item $r_1\le r_2$ for any $r\in R$;\\[-8pt]
\item $\{r_1,r_2\}\cap\{r'_1,r'_2\}=\emptyset$ for any $r,r'\in R$;\\[-8pt]
\item for any $r,r'\in R$, the intervals $[r]$ and $[r']$ are
either disjoint or
      one of them is contained in the other.
\end{itemize}
\smallskip
Abusing notation, we will write
\begin{itemize}
\item $r\subset r'$ (resp. $r\subsetneq r'$) to say that $[r]\subset[r']$ (resp. $[r]\subsetneq[r']$);\\[-8pt]
\item $r<r'$ to say that $r_2<r'_1$;
\end{itemize}

Let $s:I\to\mathcal T(W)$ and $v:R\to W$ be arbitrary maps. We
define
\begin{equation}\label{eq:bt7:10}
\BS_c(s,v)=\{[c]\in\BS_c(s)\suchthat\forall r\in R:
c_{r_1}c_{r_1+1}\cdots c_{r_2}\in v_rK\}.
\end{equation}
This set is well-defined, as $\dot v_rK=K\dot v_r$ and is closed in $\BS_c(s)$.
Consider the maps $\xi_r:C(s)\to C/K$ defined by
$\xi_r(c)=c_{r_1}c_{r_1+1}\cdots c_{r_2}K$. Then $\BS_c(s,v)$ is
the image of the intersections
$$
C(s,v)=\bigcap_{r\in R}\xi_r^{-1}(v_rK)
$$
under the quotient map $C(s)\to\BS_c(s)$.

Let $F\subset R$ be a nonempty subset such that the intervals $[f]$, where $f\in F$,
are pairwise disjoint. In what follows, we use the notation
$$
F=\{f^1,\ldots,f^\n\}
$$
with elements written in the increasing order:
$f^1<f^2<\cdots<f^\n$.
We set
$$
I^F=I\setminus\bigcup_{f\in F}[f],\quad R^F=R\setminus\{r\in R\suchthat\exists f\in F: r\subset f\}.
$$
Obviously, $R^F$ is a nested structure on $I^F$.

We also introduce the following auxiliary notation. Let $i\in I^F$.
We choose $m=0,\ldots,\n$ so that $f^m_2<i<f^{m+1}_1$.
In these inequalities and in what follows, we assume that $f^0_2=-\infty$ and
$f^{\n+1}_1=+\infty$. We set
$$
v^i=v_{f^1}v_{f^2}\cdots v_{f^m},\quad
\dot v^i=\dot v_{f^1}\dot v_{f^2}\cdots\dot v_{f^m}.
$$
Clearly, $v^i$ is image of $\dot v^i$ under the quotient
homomorphism $\mathcal N\to W$.
We define the sequences $s^F:I^F\to\mathcal T(W)$ and $v^F:R^F\to W$ by
$$
s^F_i=v^is_i(v^i)^{-1},\quad v^F_r=v^{r_1}v_r(v^{r_2})^{-1}.
$$


\begin{definition}
A sequence $c\in C(s,v)$ is called $F$-balanced if
\begin{equation}\label{eq:coh:2.5}
c_{f_1}c_{f_1+1}\cdots c_{f_2}=\dot v_f
\end{equation}
for any $f\in F$.
\end{definition}

It is easy to prove that for any $a\in\BS_c(s,v)$, there exists some $F$-balanced $c\in C(s,v)$
such that $a=[c]$.
For an $F$-balanced $c\in C(s,v)$, we define the sequence $c^F\in C(s^F)$ by
$$
c^F_i=\dot v^ic_i(\dot v^i)^{-1}.
$$
Then we set $p^F([c])=[c^F]$. By~\cite[Lemma~4]{NFBBSV}, the map $p^F$ is a well-defined,
$K$-equivariant and continuous map from $\BS_c(s,v)$ to $\BS_c(s^F,v^F)$. We call it the {\it projection
along} $F$.

A nested structure $R$ on $I$ is called {\it closed} if $I\ne\emptyset$ and
it contains the pair $(\min I,\max I)$, which we denote by $\span I$.

Let $f\in F$. For any sequence $\gamma$ on $I$, we denote by $\gamma_f$
the restriction of $\gamma$ to $[f]$. We also set
$$
R_f=\{r\in R\suchthat r\subset f\},\quad v_f=v|_{R_f}.
$$
From this construction, it is obvious that $R_f$ is a closed nested structure on $[f]$.
By~\cite[Theorem~7]{NFBBSV},
the map $p^F:\BS_c(s,v)\to\BS_c(s^F,v^F)$ is a fibre bundle with fibre
$\BS_c(s_{f^1},v_{f^1})\times\cdots\times\BS_c(s_{f^{\n}},v_{f^{\n}})$.

\subsection{Affine pavings}\label{Affine_pavings} We say that a topological space $X$ has an {\it affine paving} if there exists a filtration
$$
\emptyset=X_0\subset X_1\subset X_2\subset\cdots\subset X_{n-1}\subset X_n=X
$$
such that all $X_i$ are closed and each difference $X_i\setminus X_{i-1}$ is homeomorphic to $\C^{m_i}$
for some integer $m_i$.

Let $r\in R$ and $r^1,\ldots,r^\mathfrak q$ be the maximal (with respect to the inclusion)
elements of those elements of $R$ that are strictly contained in $r$. We write them in the increasing order
$r^1<\cdots<r^\mathfrak q$.
We set
$$
I(r,R)=[r]\setminus\bigcup_{r'\in R\atop r'\subsetneq r}[r']=[r]\setminus\bigcup_{m=1}^{\mathfrak q}[r^m].
$$
We define the sequence  $s^{(r,v)}:I(r,R)\to\mathcal T(W)$ by
$$
s^{(r,v)}_i=v_{r^1}\cdots v_{r^m}s_i(v_{r^1}\cdots v_{r^m})^{-1},
$$
where $r_2^m<i<r_1^{m+1}$. Here we suppose as usually that $r_2^0=-\infty$ and $r_1^{\q+1}=+\infty$.

\begin{definition}\label{definition:3}
Let $s$ be a sequence of reflections on $I$ and $v$ be a map from a nested structure $R$ on $I$
to $W$. We say that the pair $(s,v)$ is {\it of gallery type} if for any $r\in R$,
the sequence $s^{(r,v)}$ is of gallery type.
\end{definition}

The following results are Lemma~11 and Corollary~12 from~\cite{NFBBSV}.

\begin{proposition}\label{proposition:bt10:1}
If $(s,v)$ is of gallery type, then $(s^F,v^F)$ and $(s_{f^1},v_{f^1}),\ldots,(s_{f^{\n}},v_{f^{\n}})$
are also of gallery type. Therefore, $\BS_c(s,v)$ has an affine paving.
\end{proposition}


\section{Equivariant cohomology}\label{Equivariant cohomology}

\subsection{Definitions}\label{Def_equiv}
Let
$p_T:E_T\to B_T$
be a universal principal $T$-bundle. For any $T$-space $X$, we consider the {\it Borel construction}
$X\times_TE_T=(X\times E_T)/T$, where $T$ acts on the Cartesian product diagonally: $t(x,e)=(tx,te)$.
Then we define the {\it $T$-equivariant cohomology of $X$ with coefficients $\k$} by
$$
H^\bullet_T(X,\k)=H^\bullet(X\times_TE_T,\k).
$$
This definition a priory depends on the choice of a universal principal $T$-bundle.
However, all these cohomologies are isomorphic. Indeed let
$p'_T:E'_T\to B'_T$
be another universal principal $T$-bundle. We consider the following diagram (see~\cite[1.4]{Jantzen}):
\begin{equation}\label{eq:bt6:2.5}
\begin{tikzcd}
{}&(X\times E_T\times E'_T)/T\arrow{rd}{p_{13}/T}\arrow{ld}[swap]{p_{12}/T}&\\
X\times_TE_T&&X\times_TE'_T
\end{tikzcd}
\end{equation}
where the $T$ acts diagonally on $X\times E_T\times E'_T$. 
As $p_{12}/T$ and $p_{13}/T$ are fibre bundles with fibres $E'_T$
and $E_T$ respectively, we get by the Vietoris-Begle mapping theorem
the following diagram for cohomologies:
\begin{equation}\label{eq:bt6:2}
\begin{tikzcd}
{}&H^\bullet((X\times E_T\times E'_T)/T,\k)&\\
H^\bullet(X\times_TE_T,\k)\arrow{ru}[swap]{\sim}{(p_{12}/T)^*}&&H^\bullet(X\times_TE'_T,\k)\arrow{lu}{\sim}[swap]{(p_{13}/T)^*}
\end{tikzcd}
\end{equation}
It allows us to identify the cohomologies $H^\bullet(X\times_TE_T,\k)$ and $H^\bullet(X\times_TE'_T,\k)$.
The reader can easily check that this identification respects composition and
is identical if both universal principal bundles are equal.
%

Similar constructions are possible for the compact torus $K$: for a $K$-space $X$, we define
$$
H^\bullet_K(X,\k)=H^\bullet(X\times_KE_K,\k),
$$
where
$p_K:E_K\to B_K$
is a universal principal $K$-bundle. For any $T$-space, we have $H^\bullet_T(X,\k)\cong H^\bullet_K(X,\k)$,
see for example~\cite[1.6]{Jantzen}.

\subsection{Stiefel manifolds}\label{Stiefel manifolds} Being a Chevalley group, the group $G$ admits an embedding $G\le\GL(V)$
for some faithful representation $V$ of the Lie algebra of $G$. Let $V=V_1\oplus\cdots\oplus V_k$
be a decomposition into a direct sum of irreducible $G$-modules. By~\cite[Chapter 12]{Steinberg}, there exist for each $i$ a positive
definite Hermitian form $\<,\>_{V_i}$ on $V_i$ such that $\<gu,v\>_{V_i}=\<u,\sigma(g^{-1})v\>_{V_i}$ for any $x\in G$
and $u,v\in V_i$, where $\sigma$ is the automorphism of $G$ defined in Section~\ref{Compact subgroups of complex algebraic groups}.
Their direct sum $\<,\>_{V}$ is a positive definite Hermitian form on $V$ satisfying the same property.
We consider the unitary group
$$
U(V)=\{g\in\GL(V)\suchthat\<gu,gv\>_{V}=\<u,v\>_V\;\forall u,v\in{V}\},
$$
for which we have $U(V)\cap G=C$. Choosing an orthonormal basis $(v_1,\ldots,v_{\mathfrak r})$ of $V$,
we get an isomorphism $U(V)\stackrel\sim\to U(\mathfrak r)$ taking an operator of $U(V)$ to its matrix in this basis.

For any natural number $N\ge\r$,
we consider the {\it Stiefel manifold}\footnote{This space is usually denoted by $V_\r(\C^N)$ or $\C V_{N,\r}$.
We also transpose matrices, as we want to have a left action of $U(\r)$}
$$
E^N=\{A\in M_{\r,N}(\C)\suchthat A\bar A^T=I_\r\},
$$
where $M_{\r,N}(\C)$ is the space of $\r\times N$ matrices with respect to metric topology and $I_\r$ is the identity
matrix. The group $U(\r)$ of unitary $\r\times \r$ matrices acts on $E^N$ on the left by multiplication.
Similarly, $U(N)$ acts on $E^N$ on the right. The last action is transitive and both actions commute.
The quotient space $\Gr^N=E^N/U(\r)$ is called a {\it Grassmanian} and the corresponding
quotient map $E^N\to\Gr^N$ is a principal $U(\r)$-bun\-dle. Note that the group $U(N)$ also acts on $\Gr^N$
by the right multiplication. For $N'>N$, we get the embedding $E^N\hookrightarrow E^{N'}$ by adding $N'-N$
zero columns to the right.

Taking the direct limits
$$
E^\infty=\dlim E^N,\quad \Gr=\dlim\Gr^N,
$$
we get a universal principal $U(\r)$-bundle $E^\infty\to\Gr$.

We need the spaces $E^N$ to get the principal $K$-bundles $E^N\to E^N/K$.
It is easy to note that this bundle for $N=\infty$ is the direct limit of the bundles for $N<\infty$.

Note that for any $K$-space $X$, the group $U(N)$ acts on $X\times_KE^N$ on the right by $K(x,e)g=K(x,eg)$,
where $x\in X$, $e\in E^N$ and $g\in U(N)$.

\subsection{Equivariant cohomology of a point}\label{Equivariant_cohomology_of_a_point}
We denote by $S=H^\bullet_T(\pt,\k)$ the equivariant cohomology of a point.
It is well known that $S$ is a polynomial ring with zero odd degree component. More exactly, let $\mathfrak X(T)$ be the group of all
continuous homomorphisms $T\to\C^\times$. For each $\lm\in\mathfrak X(T)$, let $\C_\lm$ be the $\C T$-module that is equal
to $\C$ as a vector space and has the following  $T$-action: $tc=\lm(t)c$.
Then we have the line bundle $\C_\lm\times_T E_T\to B_T$ denoted by $\mathcal L_T(\lm)$,
where $E_T\to B_T$ is a universal principal $T$-bundle. We get the map ${\mathfrak X}(T)\to H^2_T(\pt)=H^2(B_T,\k)$ given by
$\lm\mapsto c_1(\mathcal L_T(\lm))$, where $c_1$ denotes the first Chern class,
which extends to the isomorphism with the symmetric algebra:
\begin{equation}\label{eq:pt:1}
\mathop{\mathrm{Sym}}({\mathfrak X}(T)\otimes_\Z\C)\ito S.
\end{equation}

Similarly, let $\mathfrak X(K)$ be the group of continuous homomorphisms $K\to\C^\times$.
For each $\lm\in\mathfrak X(K)$, we have the bundle $\mathcal L_K(\lm)$ similar to $\mathcal L_T(\lm)$.
Therefore, we have the isomorphism
\begin{equation}\label{eq:pt:1.5}
\mathop{\mathrm{Sym}}(\mathfrak X(K)\otimes_\Z\C)\ito S
\end{equation}
induced by $\lm\mapsto c_1(\mathcal L_K(\lm))$. In what follows, we identify $\mathfrak X(T)$ with $\mathfrak X(K)$
via the restriction. Then both isomorphisms~(\ref{eq:pt:1}) and~(\ref{eq:pt:1.5}) become equal.
Note that the Weyl group $W$ acts on $\mathfrak X(K)$ and $\mathfrak X(T)$ by $(w\lm)(t)=\lm(\dot w^{-1}t\dot w)$.

We are free to choose a universal principal $K$-bundle $E_K\to B_K$ to compute $S=H^\bullet(B_K,\k)$.
We assume that the $K$-action on $E_K$ can be extended to a continuous $C$-action.
The quotient map $E^\infty\to E^\infty/K$ is an example of such a bundle.
The map $\rho_w:E_K\to E_K$ defined by $\rho_w(e)=\dot we$ factors through the action
of $K$ and we get the map $\rho_w/K:E_K/K\to E_K/K$.
This map induces the ring homomorphism $(\rho_w/K)^*:S\to S$.
It is easy to check that the pullback of $\mathcal L_K(\lm)$ along $\rho_w/K$ is $\mathcal L(w^{-1}\lambda)$.
Therefore, under identification~(\ref{eq:pt:1.5}), we get
\begin{equation}\label{eq:pt:2}
(\rho_w/K)^*(u)=w^{-1}u
\end{equation}
for any $u\in S$.

For a finite space $X$ with the discrete topology and trivial action of $K$,
we identify $H^\bullet_T(X,\k)$
with $S(X)$.
More exactly, let $x\in X$ be an arbitrary point. Consider the map $j_x:E_K/K\to X\times_KE_K$
given by $j_x(Ke)=K(x,e)$. Then any element $h\in H^n_K(X,\k)$ is identified with the function
$x\mapsto j_x^*h$. A similar identification is possible for a finite discrete space $X$ with the trivial action of $T$.

\subsection{Twisted actions of $S$}\label{Twisted action of S} Suppose that $E_K\to B_K$ is a universal principal $K$-bundle
such that the $K$-action on $E_K$ can be extended to a continuous $C$-action.
An example of such a bundle is the quotient map $E^\infty\to E^\infty/K$ (see, Section~\ref{Stiefel manifolds}).
Let $s:I\to\mathcal T(W)$ be a sequence.
We assume that $I$ is embedded into another totaly ordered set $J$
(this assumption will be used in Section~\ref{Embeddings_of_the_Borel_constructions}).
Let $j$ be an element of $J\cup\{-\infty\}$. We define the map $\Sigma(s,j,w):\BS_c(s)\times_KE_K\to E_K/K$ by
\begin{equation}\label{eq:bt9:2}
K([c],e)\mapsto K(c^j\dot w)^{-1}e.
\end{equation}
The reader can easily check that this map is well-defined and continuous.
We call the map $\Sigma(s,-\infty,1)$ the {\it canonical} projection from
$\BS_c(s)\times_KE_K$ to $E_K/K$.

Taking cohomologies, we get the map
$$
\Sigma(s,j,w)^*:S\to H_K^\bullet(\BS_c(s),\k).
$$
This map induces the action of $S$ on $H^\bullet(\BS_c(s),\k)$ by
$u\cdot h=\Sigma(s,j,w)^*(u)\cup h$. For $w=1$ and $j=-\infty$, we get the canonical
action of $S$.
%
Note that these actions are independent of the choice of the universal principal $K$-bundle
and of the action of $C$.

We also consider the finite dimensional version of these maps. Let $\Sigma^N(s,j,w):\BS_c(s)\times_KE^N_K\to E^N_K/K$
be the map given by~(\ref{eq:bt9:2}). Here $N$ may be an integer greater than or equal to $\r$ or $\infty$.
Note that $\Sigma^\infty(s,j,w)$ is a representative of $\Sigma(s,j,w)$.

\section{Tensor products}\label{Tensor products} From now on, we assume that $\k$ is a commutative
ring having a finite global dimension.
We denote by $D_\k$ the derived category of the category of $\k$-modules and by $D_\k(X)$
the derived category of the category of sheaves of $\k$-modules on a topological space $X$.
We consider the bounded-below, bounded-above and bounded subcategories
$D_\k^+, D_\k^-$, $D_\k^b$ of $D_\k$, respectively. Moreover, we set $D_\k^\emptyset=D_\k$.
The similar notation is used for $D_\k(X)$.

\subsection{Tensor products of complexes of modules}

The tensor product $N\otimes_\k M$
of two left $\k$-modules is again a left $\k$-module. Similarly
the tensor product $C^\bullet\otimes_\k D^\bullet$ of two
complexes of left $\k$-modules is again a complex of left
$\k$-modules. If the ring $\k$ is obvious from the context, then
we write $\otimes$ instead of $\otimes_\k$. For any integer $n$,
there is the natural homomorphism
\begin{equation}\label{eq:coh:1}
\bigoplus_{i+j=n}H^i(C^\bullet)\otimes H^j(D^\bullet)\to
H^n(C^\bullet\otimes D^\bullet).
\end{equation}
We would like to know if this homomorphism is an isomorphism for
$C^\bullet$ and $D^\bullet$ satisfying certain conditions.

\begin{lemma}\label{lemma:coh:1} Let $D^\bullet$ be a complex of projective $\k$-modules bounded above such that $H^j(D^\bullet)$
is projective for any $j$. Then~(\ref{eq:coh:1}) is an
isomorphism.
\end{lemma}
\begin{proof} The proof is the same as that of~\cite[Theorem 3B.5]{Hatcher_AT} with the exception
of the case where the differentials of $C^\bullet$ are all zero.
In that case, the complex $C^\bullet\otimes D^\bullet$ is a direct
sum of complexes $C^i[-i]\otimes D^\bullet$. So it suffices to
consider the case where $C^\bullet$ is concentrated in degree
zero. We note that the complex $D^\bullet$ is isomorphic to the
following one:
$$
\cdots\to B^i\oplus H^i(D^\bullet)\oplus
B^{i+1}\stackrel{\partial^i}\to B^{i+1}\oplus
H^{i+1}(D^\bullet)\oplus B^{i+2}\to\cdots,
$$
where $B^i$ is the $i$th coboundary and $\partial^i$ maps the
first to summands to zero and the third one identically. As the
tensor product is distributive over the direct sum, the required
result follows.
\end{proof}

For every complex $X\in\Ob(D_\k^*)$ where $*=-,+,b,\emptyset$, there exist
a complex of projective modules $P\in\Ob(D_\k^*)$ and
a quasi-isomorphism $P\to X$. This allows us to define the derived
tensor product:
$$
\Lotimes:D_\k^*\times D_\k^*\to D_\k^*.
$$
Thus we get the map
\begin{equation}\label{eq:coh:2}
\bigoplus_{i+j=n}H^i(C^\bullet)\otimes H^j(D^\bullet)\to
H^n(C^\bullet\Lotimes D^\bullet).
\end{equation}
for any complexes $C^\bullet,D^\bullet\in\Ob(D_\k^*)$.
We will need the following result.

\begin{lemma}\label{lemma:coh:2}
Let $C^\bullet,D^\bullet\in\Ob(D_\k^*)$ be such that
$H^i(C^\bullet)=0$ for $i<0$ and $H^j(D^\bullet)$ is projective
for $j\le N$. Then~(\ref{eq:coh:2}) is an isomorphism for $n\le
N-\gld(\k)$.
\end{lemma}
\begin{proof} Consider the following distinguished triangle:
$$
\tau_{\le N}D^\bullet\to
D^\bullet\to\tau_{>N}D^\bullet\stackrel+\to,
$$
where $\tau_{\le N}$ and $\tau_{>N}$ are the standard truncation
functors. Tensoring with $C^\bullet$, we get another distinguished
triangle:
$$
C^\bullet\Lotimes\tau_{\le N}D^\bullet\to C^\bullet\Lotimes
D^\bullet\to C^\bullet\Lotimes\tau_{>N}D^\bullet\stackrel+\to
$$
Hence we get an exact sequence
\begin{equation}\label{eq:coh:3}
H^{n-1}(C^\bullet\Lotimes\tau_{>N}D^\bullet)\to
H^n(C^\bullet\Lotimes\tau_{\le N}D^\bullet)\to
H^n(C^\bullet\Lotimes D^\bullet)\to
H^n(C^\bullet\Lotimes\tau_{>N}D^\bullet).
\end{equation}

Without loss of generality, we can assume that $C^i=0$ for $i<0$.
Then there exists a complex $P^\bullet$ of projective modules and
a quasi-isomorphism $P^\bullet\to C^\bullet$ such that $P^i=0$ for
$i<-\gld(\k)$. This follows, for example, from the
Cartan-Eilenberg resolution. We get
$$
(C^\bullet\Lotimes\tau_{>N}D^\bullet)^n=P^\bullet\otimes\tau_{>N}D^\bullet=\bigoplus_{i+j=n\atop
i\ge-\gld(\k),j>N}P^i\otimes D^j.
$$
The last sum is obviously zero if $n\le N-\gld(\k)$. In thais
case, we get from~(\ref{eq:coh:3}) the commutative diagram
$$
\begin{tikzcd}
H^n(C^\bullet\Lotimes\tau_{\le N}D^\bullet)\arrow{r}{\sim}&H^n(C^\bullet\Lotimes D^\bullet)\\
\displaystyle\bigoplus_{i+j=n}H^i(C^\bullet)\otimes H^j(\tau_{\le
N}D^\bullet)\arrow{r}\arrow{u}&\displaystyle\bigoplus_{i+j=n}H^i(C^\bullet)\otimes
H^j(D^\bullet)\arrow{u}
\end{tikzcd}
$$
The left vertical arrow is an isomorphism by
Lemma~\ref{lemma:coh:1} and the bottom horizontal arrow is an
isomorphism as the summation in both formulas can be restricted to
$j\le n\le N$. Hence the right vertical arrow is also an
isomorphism.
\end{proof}

\subsection{Projection formula and base change for proper maps}\label{Projection_formula_and_base_change_for_proper_maps}
Let $f:X\to Y$ be a proper map between locally compact topological spaces,
$\F^\bullet\in\Ob(D_\k(X)^+)$ and $\G^\bullet\in\Ob(D_\k(Y)^+)$.
Then there exists an isomorphism
$$
Rf_*\F^\bullet\Lotimes\G^\bullet\stackrel\sim\to Rf_*(\F^\bullet\Lotimes
f^*\G^\bullet),
$$
which is called the {\it projection formula}. It follows from the
proof of this formula, for example~\cite[Proposition 2.6.6]{KS},
that (in this special case where $f$ is proper) it is actually
conjugate to the natural map
$f^*Rf_*\F^\bullet\Lotimes f^*\G^\bullet\to\F^\bullet\Lotimes f^*\G^\bullet$
(see the proof of~\cite[Proposition 2.5.13]{KS} and the proof of formula (2.3.21)
from the same book).

We have a similar description for the base change morphism. Let
that $X$ and $Y$ be compact topological spaces and
$\F^\bullet\in\Ob(D_\k(X)^+)$. We consider the Cartesian product
$X\times Y$ and the following natural maps:
\begin{equation}\label{eq:coh:5}
\begin{tikzcd}
{}&X\times Y\arrow{rd}{p_Y}\arrow{ld}[swap]{p_X}\arrow{dd}{a}&\\
X\arrow{rd}[swap]{a_X}&&Y\arrow{ld}{a_Y}\\
&\pt&
\end{tikzcd}
\end{equation}
Let
\begin{equation}\label{eq:coh:4}
a_Y^*Ra_{X*}\F^\bullet\to Rp_{Y*}p_X^*\F^\bullet
\end{equation}
be the morphism conjugate to the morphism
$p_Y^*a_Y^*Ra_{X*}\F^\bullet=p_X^*a_X^*Ra_{X*}\F^\bullet\to
p_X^*\F^\bullet$ induced by the counit $a_X^*Ra_{X*}\to\id$. To
prove that~(\ref{eq:coh:4}) is an isomorphism, we can argue
locally and reduce the problem to the case $Y=\pt$. In that case,
the result follows from the zigzag identity for the unit and
counit.

\subsection{Cross product and the K\"unneth formula} Let $X$
be a topological space and $\F^\bullet,\G^\bullet\in\Ob(D_\k(X)^+)$. 
Consider the morphism
$$
\cup_X:Ra_{X*}\F^\bullet\Lotimes Ra_{X*}\G^\bullet\to
Ra_{X*}(\F^\bullet\Lotimes\G^\bullet)
$$
that is conjugate to the product of counits
$$
a_X^*(a_{X*}\F^\bullet\Lotimes
a_{X*}\G^\bullet)=a_X^*a_{X*}\F^\bullet\Lotimes
a_X^*a_{X*}\G^\bullet\to\F^\bullet\Lotimes\G^\bullet.
$$
Taking cohomologies and applying morphism~(\ref{eq:coh:2}), we get
the map $\bigoplus_{i+j=n}{\mathbb
H}^i(X,\F^\bullet)\otimes{\mathbb H}^i(X,\G^\bullet)\to{\mathbb
H}^n(X,\F^\bullet\Lotimes\G^\bullet)$, which is called the {\it
cup product}~\cite[Exercise II.17]{KS}.

Now let that $X$ and $Y$ be topological spaces and
$\F^\bullet\in\Ob(D_\k(X)^+)$ and $\G^\bullet\in\Ob(D_\k(Y)^+)$.
Then we have diagram~(\ref{eq:coh:5}).
%
If we additionally assume that $X$ and $Y$ are compact, then we
have the following diagram:
\begin{equation}\label{eq:coh:9}
\begin{tikzcd}
Ra_*p_X^*\F^\bullet\Lotimes Ra_*p_Y^*\G^\bullet\arrow{rr}{\cup_{X\times Y}}&&Ra_*(p_X^*\F^\bullet\Lotimes p_Y^*\G^\bullet)\\
&Ra_{X*}\F^\bullet\Lotimes
Ra_{Y*}\G^\bullet\arrow{ul}\arrow{ur}[swap]{\sim}&
\end{tikzcd}
\end{equation}
The left arrow comes from the units of adjunction $\id\to
p_{X*}p_X^*$ and $\id\to p_{Y*}p_Y^*$ and the right arrow is the
K\"unneth isomorphism. As we want to prove the commutativity of
the above diagram, we need to look more closely at this
isomorphism. Actually it is the following composition of
projection formulas and the base change:
$$
\begin{tikzcd}[column sep=15ex]
Ra_{X*}\F^\bullet\Lotimes Ra_{Y*}\G^\bullet\arrow{r}{\text{projection formula}}[swap]{\sim}&Ra_{Y*}(a_Y^*Ra_{X*}\F^\bullet\Lotimes\G^\bullet)\arrow{r}{\text{base change}}[swap]{\sim}&{}
\end{tikzcd}
$$
$$
\!\!\!\!\!\!\!\!\!\!\!\!\!\!\!\!\!\!\!\!\!\!\!\!\!\!\!\!\!\!\!\!\!\!\!
\begin{tikzcd}[column sep=15ex]
Ra_{Y*}(Rp_{Y*}p_X^*\F^\bullet\Lotimes\G^\bullet)\arrow{r}{\text{projection formula}}[swap]{\sim}&Ra_*(p_X^*\F^\bullet\Lotimes Rp_Y^*\G^\bullet).
\end{tikzcd}
$$
Now it remains to apply the constructions of
Section~\ref{Projection_formula_and_base_change_for_proper_maps}.

Let $X$ and $Y$ be compact topological spaces such that $H^j(X,\k)$ is
a free $\k$-module for $j\le N$. For $n\le N-\gld(\k)$, we get the following commutative diagram:
$$
\!\!\!
\begin{tikzcd}\displaystyle\bigoplus_{i+j=n}H^i(X\times Y,\k)\otimes H^j(X\times Y,\k)\arrow{rr}{\text{cup product}}&&H^n(X\times Y,\k)\\
&\displaystyle\bigoplus_{i+j=n}H^i(X,\k)\otimes H^j(Y,\k)\arrow{ul}\arrow{ur}[swap]{\sim}&
\end{tikzcd}
$$

This construction can be easily iterated as follows. Let $X_1,\ldots,X_m$ be topological spaces such
that $H^n(X_j,\k)$ are free for all $n\le N$ and $j=2,\ldots,m$. Then for any $n\le N-\gld(\k)$ we have the isomorphism
\begin{equation}\label{eq:coh:12}
\bigoplus_{i_1+\cdots+i_m=n}H^{i_1}(X_1,\k)\otimes\cdots\otimes H^{i_m}(X_m,\k)\stackrel\sim\to H^n(X_1\times\cdots\times X_m,\k)
\end{equation}
that is given by the cross product $a_1\otimes\cdots\otimes a_m\mapsto p_1^*(a_1)\cup\cdots\cup p_m^*(a_m)$,
where $p_i:X_1\times\cdots\times X_m\to X_i$ is the projection to the $i$th coordinate.

\subsection{Embeddings of the Borel constructions}\label{Embeddings_of_the_Borel_constructions}
 We return here to the notation of Section~\ref{Nested_fibre_bundles}.
For any $m=1,\ldots,\n$ and $N\ge\r$, we define the map
$$
q^{m,N}:\BS_c(s,v)\times_K E^N\to\BS_c(s_{f^m},v_{f^m})\times_K E^N
$$
by
\begin{equation}\label{eq:bt6:15}
\begin{tikzcd}
K([c],e)\arrow[mapsto]{r}{q^{m,N}}&K\(\big[c_{f^m}\big],(c^{f_1^m-1})^{-1}e\).
\end{tikzcd}
\end{equation}
It is easy to check that this map is well-defined and continuous.
Remember also from Section~\ref{Twisted action of S} that we have
the maps
$$
\Sigma^N(s^F,f^m_1,v_{f^1}\cdots v_{f^{m-1}}):\BS_c(s^F)\times_K E^N\to E^N/K.
$$
Let $\sigma^N_m:\BS_c(s^F,v^F)\times_K E^N\to E^N/K$ be the restriction of this map. More precisely, we have
$$
\hspace{-140pt}
\begin{tikzcd}
K([d],e)\arrow[mapsto]{r}{\sigma^N_m}&
K((d_{\min I}\cdots d_{f^1_1-1}d_{f^1_2+1}\cdots d_{f^2_1-1}\cdots
\end{tikzcd}
$$
$$
\hspace{200pt}\cdots d_{f^{m-1}_2+1}\cdots d_{f^m_1-1}\dot v_{f^1}\cdots\dot v_{f^{m-1}})^{-1}e.
$$
%
%
We consider the map
$$
\mu^N=(p^F\times_K\id)\boxtimes q^{1,N}\boxtimes\cdots\boxtimes q^{\n,N}
$$
from $BS_c(s,v)\times_K E^N$ to
$$
M^N=(\BS_c(s^F,v^F)\times_K E^N)\times\prod_{m=1}^\n\BS_c(s_{f^m},v_{f^m})\times_K E^N.
$$
We abbreviate $M=M^\infty$, $\mu=\mu^\infty$ and $\sigma_m=\sigma_m^\infty$.

The following result allows us to embed Borel constructions to Cartesian product of Borel constructions
in a way similar to
~\cite[Lemma~21]{NFBBSV}.

\begin{lemma}\label{lemma:10}
For $N<\infty$, the map $\mu^N$ is a topological embedding.
Its image consists of all $\n+1$-tuples $\big(K([d],e),K([c^{(1)}],e_1),\ldots,K([c^{(\n)}],e_\n)\big)$ such that
\begin{equation}\label{eq:13}
Ke_m=\sigma^N_m(K([d],e))
\end{equation}
for any $m=1,\ldots,\n$.
\end{lemma}
\begin{proof} {\it Part 1: $\mu^N$ is a topological embedding.}
As $\mu^N$ maps a compact space to a Hausdorff space, it suffices to prove that it is injective.
Let $K([c],e)$ and $K([d],h)$ be two orbits of $\BS_c(s,v)\times_K E^N$ mapped to the same $\n+1$-tuple by $\mu^N$.
We can assume that $c$ and $d$ are $F$-balanced.
It follows that there exist some elements $\tilde k_0,\tilde k_1,\ldots,\tilde k_\n\in K$ such that
$$
([d^F],h)=\tilde k_0^{-1}([c^F],e),\quad
\(\big[d_{f^m}\big],(d^{f_1^m-1})^{-1}h\)=\tilde k_m^{-1}\(\big[c_{f^m}\big],(c^{f_1^m-1})^{-1}e\)
$$
for $m=1,\ldots,\n$. Hence in its turn, it follows that
%
\begin{equation}\label{eq:8}
h=\tilde k_0^{-1}e,\quad
(d^{f_1^m-1})^{-1}h
=
\tilde k_m^{-1}
(c^{f_1^m-1})^{-1}e
\end{equation}
and that there exist some sequences
$k^{(0)}\in K(I^F)$, $k^{(1)}\in K([f^1])$ , \ldots, $k^{(\n)}\in K([f^{\n}])$
such that
\begin{equation}\label{eq:9}
d^F=\tilde k_0^{-1}c^Fk^{(0)},\quad d_{f^m}=\tilde k_m^{-1}c_{f^m}k^{(m)}
\end{equation}
for any $m=1,\ldots,\n$.

The fact that $c$ and $d$ are $F$-balanced and the second equality of~(\ref{eq:9}) immediately prove that
\begin{equation}\label{eq:10}
\dot v_{f^m}=\tilde k_m^{-1}
\dot v_{f^m}
k^{(m)}_{f^m_2}
\end{equation}
for any $m=1,\ldots,\n$. For our calculations below, it is convenient to define first $k^{(0)}_{-\infty}=\tilde k_0$ and
then for any $m=1,\ldots,\n$ to define $k^{(0)}_{f^m_1-1}=k^{(0)}_j$,
where
$j$ is the maximal element of $ I^F\cup\{-\infty\}$ less than or equal to $f^m_1-1$.

We will prove by induction on $m=1,\ldots,\n$ that
\begin{equation}\label{eq:11}
\tilde k_m=(\dot v_{f^1}\cdots\dot v_{f^{m-1}})^{-1}
k^{(0)}_{f^m_1-1}\dot v_{f^1}\cdots\dot v_{f^{m-1}}.
\end{equation}
First consider the case $m=1$ and $f^1_1=\min I$. The second equality of~(\ref{eq:8}) now takes the form $h=\tilde k_1^{-1}e$.
Comparing it with the first equality of~(\ref{eq:8}), we get
\begin{equation}\label{eq:12}
\tilde k_1=\tilde k_0=k^{(0)}_{-\infty}=k^{(0)}_{f^1_1-1}.
\end{equation}
Now let $m=1$ and $f^m_1>\min I$. From the first equality of~(\ref{eq:9}), we get
$$
d^{f_1^1-1}=\tilde k_0^{-1}c^{f_1^1-1}k^{(0)}_{f^1_1-1}.
$$
Substituting this to the second equality of~(\ref{eq:8}) for $m=1$, we get
$$
(k^{(0)}_{f^1_1-1})^{-1}(c^{f_1^1-1})^{-1}\tilde k_0h
=
\tilde k_1^{-1}
(c^{f_1^1-1})^{-1}e
$$
Hence and from the first equality of~(\ref{eq:8}), we get $k^{(0)}_{f^1_1-1}=\tilde k_1$.

Now suppose that $m>1$ and for smaller indices~(\ref{eq:11}) is true. Applying the second equation of~(\ref{eq:9}), we get
\begin{equation}\label{eq:bt6:3}
\begin{array}{l}
\!\!\!\!\!
d^{f_1^m-1}=d^{f_1^1-1}
\Big(\tilde k_1^{-1}c_{f^1_1}\cdots c_{f^1_2}k^{(1)}_{f^1_2}\Big)\times\\
\quad\quad\times d_{f^1_2+1}\cdots d_{f_1^2-1}\Big(\tilde k_2^{-1}c_{f^2_1}\cdots c_{f^2_2}k^{(2)}_{f^2_2}\Big)
\cdots
d_{f^{m-2}_2+1}\cdots d_{f_1^{m-1}-1}\times\\
\quad\quad\quad\quad\quad\quad\quad\quad\quad\quad\quad\quad
\times\Big(\tilde k_{m-1}^{-1}c_{f^{m-1}_1}\cdots c_{f^{m-1}_2}k^{(m-1)}_{f^{m-1}_2}\Big)
d_{f^{m-1}_2+1}\cdots d_{f_1^m-1}.
\end{array}
\end{equation}
Here we replaced elements of $d$ with the corresponding elements of $c$ in the brackets.
To replace the remaining elements of $d$ with elements of $c$, we apply the first equation of~(\ref{eq:9}):
\begin{equation}\label{eq:bt6:7}
d^{f_1^1-1}=(d^F)^{f_1^1-1}=\tilde k_0^{-1}(c^F)^{f_1^1-1}k^{(0)}_{f_1^1-1}
=\tilde k_0^{-1}c^{f_1^1-1}\tilde k_1
\end{equation}
and
$$
d_{f^h_2+1}^F\cdots d_{f_1^{h+1}-1}^F=(k^{(0)}_{f^h_1-1})^{-1}c^F_{f^h_2+1}\cdots c^F_{f_1^{h+1}-1}k^{(0)}_{f^{h+1}_1-1}
$$
for any $h=1,\ldots,m-1$. Conjugating both sides of the last equality by $(\dot v_{f^1}\cdots\dot v_{f^h})^{-1}$, we get
\begin{equation}\label{eq:bt6:4}
\begin{array}{l}
\!\!\!\!\!
d_{f^h_2+1}\cdots d_{f_1^{h+1}-1}
=(\dot v_{f^1}\cdots\dot v_{f^h})^{-1}(k^{(0)}_{f^h_1-1})^{-1}\dot v_{f^1}\cdots\dot v_{f^h}\times\\[6pt]
\quad\quad\quad\quad\quad\quad\quad\quad\quad\times c_{f^h_2+1}\cdots c_{f_1^{h+1}-1}(\dot v_{f^1}\cdots\dot v_{f^h})^{-1}k^{(0)}_{f^{h+1}_1-1}\dot v_{f^1}\cdots\dot v_{f^h}.
\end{array}
\end{equation}
We claim that
\begin{equation}\label{eq:bt6:5}
(\dot v_{f^1}\cdots\dot v_{f^h})^{-1}(k^{(0)}_{f^h_1-1})^{-1}\dot v_{f^1}\cdots\dot v_{f^h}=k^{(h)}_{f_2^h}
\end{equation}
for $1\le h\le m-1$ and that
\begin{equation}\label{eq:bt6:6}
(\dot v_{f^1}\cdots\dot v_{f^h})^{-1}k^{(0)}_{f^{h+1}_1-1}\dot v_{f^1}\cdots\dot v_{f^h}=\tilde k_{h+1}.
\end{equation}
$1\le h<m-1$. The second equality follows from the inductive hypothesis.
Conjugating the first equality by $\dot v_{f^h}$, we get an equivalent equality
$$
(\dot v_{f^1}\cdots\dot v_{f^{h-1}})^{-1}(k^{(0)}_{f^h_1-1})^{-1}\dot v_{f^1}\cdots\dot v_{f^{h-1}}=\dot v_{f^h}k^{(h)}_{f_2^h}(\dot v_{f^h})^{-1},
$$
which is equivalent to
$$
\tilde k_h=\dot v_{f^h}k^{(h)}_{f_2^h}(\dot v_{f^h})^{-1}
$$
by the inductive hypothesis. This equality follows from~(\ref{eq:10}).
Substituting~(\ref{eq:bt6:5}) and~(\ref{eq:bt6:6}) to~(\ref{eq:bt6:4}), we get
$$
d_{f^h_2+1}\cdots d_{f_1^{h+1}-1}=(k^{(h)}_{f_2^h})^{-1}c_{f^h_2+1}\cdots c_{f_1^{h+1}-1}\tilde k_{h+1}
$$
for $1\le h<m-1$ and
$$
d_{f^{m-1}_2+1}\cdots d_{f_1^m-1}=(k^{(m-1)}_{f_2^{m-1}})^{-1}c_{f^{m-1}_2+1}\cdots c_{f_1^m-1}(\dot v_{f^1}\cdots\dot v_{f^{m-1}})^{-1}k^{(0)}_{f^m_1-1}\dot v_{f^1}\cdots\dot v_{f^{m-1}}
$$
Substituting these two equations and~(\ref{eq:bt6:7}) to~(\ref{eq:bt6:3}), we get
$$
d^{f_1^m-1}=\tilde k_0^{-1}c^{f_1^m-1}(\dot v_{f^1}\cdots\dot v_{f^{m-1}})^{-1}k^{(0)}_{f^m_1-1}\dot v_{f^1}\cdots\dot v_{f^{m-1}}
$$
Multiplying $h$ by the inverse of both sides, we get by the first equation of~(\ref{eq:8}) that
\begin{multline*}
(d^{f_1^m-1})^{-1}h=(\dot v_{f^1}\cdots\dot v_{f^{m-1}})^{-1}(k^{(0)}_{f^m_1-1})^{-1}\dot v_{f^1}\cdots\dot v_{f^{m-1}}(c^{f_1^m-1})^{-1}\tilde k_0h\\
=(\dot v_{f^1}\cdots\dot v_{f^{m-1}})^{-1}(k^{(0)}_{f^m_1-1})^{-1}\dot v_{f^1}\cdots\dot v_{f^{m-1}}(c^{f_1^m-1})^{-1}e.
\end{multline*}
Comparing it with the second equation of~(\ref{eq:8}), we get~(\ref{eq:11}).

Now let us define the sequence $k'\in K(I)$ by
$$
k'_i=
\left\{
\begin{array}{ll}
(\dot v_{f^1}\cdots\dot v_{f^{m-1}})^{-1}k^{(0)}_i\dot v_{f^1}\cdots\dot v_{f^{m-1}}&\text{ if }f^{m-1}_2<i<f^m_1;\\[6pt]
k^{(m)}_i&\text{ if }i\in[f^m].
\end{array}
\right.
$$
It is convenient to define $k'_{-\infty}=\tilde k_0$. 

We claim that $d=\tilde k_0^{-1}ck'$. In other words, we have to prove that 
$d_i=(k'_{i-1})^{-1}c_ik'_i$ for any $i\in I$. This formula is true if both $i-1$ and $i$ belong either to $I^F$ or
to some $[f^m]$, as it follows from~(\ref{eq:9}).

First, we consider the case where $i\in I^F$ but $i-1\notin I^F$. If $i>\min I$, then $i=f^m_2+1$ for some $m$.
From the first formula of~(\ref{eq:9}), we get
$$
d^F_i=(k^{(0)}_{f^m_1-1})^{-1}c^F_ik^{(0)}_i.
$$
Conjugating this equality by $(\dot v_{f^1}\cdots \dot v_{f^m})^{-1}$ and applying~(\ref{eq:11}) and~(\ref{eq:10}), we get
\begin{multline*}
d_i=(\dot v_{f^1}\cdots \dot v_{f^m})^{-1}(k^{(0)}_{f^m_1-1})^{-1}v_{f^1}\cdots \dot v_{f^m}
c_i(\dot v_{f^1}\cdots \dot v_{f^m})^{-1}k^{(0)}_i\dot v_{f^1}\cdots \dot v_{f^m}\\
=
\dot v_{f^m}^{-1}\tilde k_m^{-1}\dot v_{f^m}c_ik'_i=(k^{(m)}_{f^m_2})^{-1}c_ik'_i=(k'_{i-1})^{-1}c_ik'_i.
\end{multline*}
Now suppose that $i=\min I$.
From the first formula of~(\ref{eq:9}), we get
$$
d_i=d^F_i=\tilde k_0^{-1}c^F_ik^{(0)}_i=\tilde k_0^{-1}c_ik^{(0)}_i=(k'_{-\infty})^{-1}c_ik'_i=(k'_{i-1})^{-1}c_ik'_i.
$$

Now consider the case $i\in[f^m]$ but $i-1\notin[f^m]$. Then we have $i=f^m_1$.
First suppose that $i-1\in I^F$. Then $f^{m-1}_2<i-1<f^m_1$. From the second equality of~(\ref{eq:9})
and by~(\ref{eq:11}), we get
$$
d_i=\tilde k_m^{-1}c_ik^{(m)}_i=(\dot v_{f^1}\cdots\dot v_{f^{m-1}})^{-1}
(k^{(0)}_{f^m_1-1})^{-1}\dot v_{f^1}\cdots\dot v_{f^{m-1}}c_ik^{(m)}_j
=(k'_{i-1})^{-1}c_ik'_i.
$$
If $i=\min I$, then $m=1$ and $f^1_1=\min I$. From the second equality of~(\ref{eq:9}) and~(\ref{eq:12}), we get
$$
d_i=\tilde k_1^{-1}c_ik^{(1)}_i=\tilde k_0^{-1}c_ik'_i=(k'_{i-1})^{-1}c_ik'_i.
$$
Finally suppose that $m>1$ and $i-1=f^{m-1}_2$. Then we have $k^{(0)}_{f^m_1-1}=k^{(0)}_{f^{m-1}_1-1}$.
From the second equality of~(\ref{eq:9}), by~(\ref{eq:11}) applied twice and by~(\ref{eq:10}),
we get
\begin{multline*}
d_i=\tilde k_m^{-1}c_ik^{(m)}_i=(\dot v_{f^1}\cdots\dot v_{f^{m-1}})^{-1}
(k^{(0)}_{f^m_1-1})^{-1}\dot v_{f^1}\cdots\dot v_{f^{m-1}}c_ik^{(m)}_i\\
=(\dot v_{f^1}\cdots\dot v_{f^{m-1}})^{-1}
(k^{(0)}_{f^{m-1}_1-1})^{-1}\dot v_{f^1}\cdots\dot v_{f^{m-1}}c_ik^{(m)}_i\\
=(\dot v_{f^{m-1}})^{-1}
(\tilde k_{m-1})^{-1}
\dot v_{f^{m-1}}c_ik^{(m)}_i=
(k^{(m-1)}_{f^{m-1}_2})^{-1}c_ik^{(m)}_i=(k'_{i-1})^{-1}c_ik'_i.
\end{multline*}
We have therefore proved that $[d]=\tilde k_0^{-1}[c]$. Hence and from the first equation of~(\ref{eq:8}), we get
$$
K([d],h)=K(\tilde k_0^{-1}[c],\tilde k_0^{-1}e)=K([c],e).
$$

{\it Part 2: the description of the image.} All elements of the image satisfy~(\ref{eq:13}), as
for any $F$-balanced $c\in C(s,v)$, we have
\begin{equation}\label{eq:bt6:8}
\begin{array}{l}\!\!\!\!
(c^{f_1^m-1})^{-1}e\\[6pt]
\quad\quad=(c_{\min I}\cdots c_{f^1_1-1}\dot v_{f^1}c_{f^1_2+1}\cdots c_{f^2_1-1}\dot v_{f^2}\cdots\dot v_{f^{m-1}}c_{f^{m-1}_2+1}\cdots c_{f^m_1-1})^{-1}e\\[6pt]
\quad\quad\quad=(c^F_{\min I}\cdots c^F_{f^1_1-1}c^F_{f^1_2+1}\cdots c^F_{f^2_1-1}\cdots c^F_{f^{m-1}_2+1}\cdots c^F_{f^m_1-1}\dot v_{f^1}\cdots\dot v_{f^{m-1}})^{-1}e.
\end{array}
\end{equation}
Now suppose on the contrary that~(\ref{eq:13}) are satisfied.
By~\cite[Lemma~6(1)]{NFBBSV},
there exists an $F$-balanced $u\in\BS_c(s,v)$ such that $p^F([u])=[d]$.
As $p^F([u])=[u^F]$, we can without loss of generality assume that $d=u^F$ (see the remark before this proof).
Let us write~(\ref{eq:13}) in the form
$$
k_me_m=(u^F_{\min I}\cdots u^F_{f^1_1-1}u^F_{f^1_2+1}\cdots u^F_{f^2_1-1}\cdots u^F_{f^{m-1}_2+1}\cdots u^F_{f^m_1-1}\dot v_{f^1}\cdots\dot v_{f^{m-1}})^{-1}e
$$
for some $k_m\in K$. By~\cite[Lemma~5]{NFBBSV}, we can write $k_m[c^{(m)}]=[z^{(m)}]$ for some $z^{(m)}\in\widetilde C(s_{f^m},v_{f^m})$.
We define
$$
c_i=
\left\{
\begin{array}{ll}
u_i&\text{ if }i\in I^F\\
z^{(m)}_i&\text{ if }i\in[f^m].
\end{array}
\right.
$$
It is easy to see that $c$ is $F$-balanced and $c^F=u^F$.
We get
$$
(p^F\times_K\id)(K([c],e))=K([c^F],e)=K([u^F],e)=K([d],e).
$$
On the other hand, for any $m=1,\ldots,\n$, applying~(\ref{eq:bt6:8}), we get
\begin{multline*}
q^{m,N}(K([c],e))=K\(\big[z^{(m)}\big],(c^{f_1^m-1})^{-1}e\)\\
=K\(\big[z^{(m)}\big],( c^F_{\min I}\cdots c^F_{f^1_1-1} c^F_{f^1_2+1}\cdots c^F_{f^2_1-1}\cdots c^F_{f^{m-1}_2+1}\cdots c^F_{f^m_1-1}\dot v_{f^1}\cdots\dot v_{f^{m-1}})^{-1}e\)\\
=K\(\big[z^{(m)}\big],( u^F_{\min I}\cdots u^F_{f^1_1-1} u^F_{f^1_2+1}\cdots u^F_{f^2_1-1}\cdots u^F_{f^{m-1}_2+1}\cdots u^F_{f^m_1-1}\dot v_{f^1}\cdots\dot v_{f^{m-1}})^{-1}e\)\\
=K(k_m[c^{(m)}],k_me_m)=K([c^{(m)}],e_m).
\end{multline*}
Hence we, get
$$
\mu^N(K([c],e))=\big(K([d],e),K([c^{(1)}],e_1),\ldots,K([c^{(\n)}],e_\n)\big).
$$
\end{proof}

\subsection{The difference}\label{Difference_MNimmuN} We now would like to study
the cohomology with compact support of the space $M^N\setminus\im\mu^N$ for $N<\infty$ with the help of spectral sequences.
To do it, we need the corresponding fibre bundles. They are given by the following lemma.

\begin{lemma}\label{lemma:11}
Let $\r\le N<\infty$ and $\kappa:M^N\setminus\im\mu^N\to\BS_c(s^F,v^F)\times_KE^N$ be the projection to the first component
and $\eta:\BS_c(s^F,v^F)\times_KE^N\to E^N/K$ be the canonical projection. 
\begin{enumerate}
\item\label{lemma:11:p:1} $\kappa$ is a fibre bundle.\\[-10pt]
\item\label{lemma:11:p:2} the composition $\eta\kappa$ is a fibre bundle.
\end{enumerate}
\end{lemma}
\begin{proof}
\ref{lemma:11:p:1}
Let $b$ be an arbitrary element (orbit) of $\BS_c(s^F,v^F)\times_KE^N$.
The right action of the unitary group $U(N)$ on $E^N$ induces a right action of $U(N)$ on $E^N/K$.
For any $m=1,\ldots,\n$, let $t_m:U(N)\to E^N/K$ be the map defined by $t_m(g)=\sigma^N_m(b)g$.
As $U(N)$ acts transitively on $E^N$, it acts transitively on $E^N/K$ and $t_m$ is a fibre bundle.
Therefore, there exists an open neighbourhood of $V_m$ of $\sigma_m(b)$ and a continuous section $g_m:V_m\to U(N)$ of $t_m$.
Hence for any $u\in V_m$, we get
\begin{equation}\label{eq:11.5}
u=t_m(g_m(u))=\sigma^N_m(b)g_m(u).
\end{equation}
We define $H=\bigcap_{m=1}^\n(\sigma^N_m)^{-1}(V_m)$.
It is an open subset of $\BS_c(s^F,v^F)\times_KE^N$ containing $b$.

We construct the map $\phi:H\times\kappa^{-1}(b)\to M^N$ by
$$
\begin{tikzcd}
\big(h,(b,a_1,\ldots,a_\n)\big)\arrow[mapsto]{r}{\phi}&\big(h,a_1g_1(\sigma^N_1(h)),\ldots,a_\n g_\n(\sigma^N_\n(h))\big)
\end{tikzcd}
$$
Here we used the right action of $U(N)$ on Borel constructions mentioned in Section~\ref{Stiefel manifolds}.
Suppose that the right-hand side of the above formula belongs to $\im\mu^N$. Let us write $a_m=K(b_m,e_m)$
for some $b_m\in\BS_c(s_{f^m},v_{f^m})$ and $e_m\in E^N$.
By Lemma~\ref{lemma:10}, we get that
$$
Ke_mg_m(\sigma^N_m(h))=\sigma^N_m(h)
$$
for any $m=1,\ldots,\n$. By~(\ref{eq:11.5}), this equality is equivalent to
$$
Ke_m=\sigma^N_m(h)g_m(\sigma^N_m(h))^{-1}=\sigma^N_m(b).
$$
Hence by Lemma~\ref{lemma:10}, we get a contradiction
$(b,a_1,\ldots,a_\n)\in\im\mu^N$.
Thus we have proved that $\im\phi\cap\im\mu^N=\emptyset$. On the other hand, it is obvious that $\im\phi\subset\kappa^{-1}(H)$.
Therefore, $\phi$ is actually a continuous map from $H\times\kappa^{-1}(b)$ to $\kappa^{-1}(H)$.

It is easy to prove that $\phi$ is a homeomorphism. Indeed the inverse map $\kappa^{-1}(H)\to H\times\kappa^{-1}(b)$
is given by
$$
\begin{tikzcd}
\big(h,a_1,\ldots,a_\n\big)\arrow[mapsto]{r}&\big(h,(b,a_1g_1(\sigma^N_1(h))^{-1},\ldots,a_\n g_\n(\sigma^N_\n(h))^{-1})\big).
\end{tikzcd}
$$
We get the following commutative diagram:
$$
\begin{tikzcd}
H\times\kappa^{-1}(b)\arrow{rr}{\phi}[swap]{\sim}\arrow{rd}[swap]{p_1}&&\kappa^{-1}(H)\arrow{ld}{\kappa}\\
&H&
\end{tikzcd}
$$
Finally note that $\kappa^{-1}(b)$ are homeomorphic for different $b$, as the space
$\BS_c(s^F,v^F)\times_KE^N$ is connected and compact.

\ref{lemma:11:p:2} Let $\bar e$ be an arbitrary point of $E^N/K$.
Let $\nu_N:E^N\to E^N/K$ denote the quotient map. As it is a principal $K$-bundle, there exists
an open neighbourhood $U\subset E^N/K$ of $\bar e$ and a homeomorphism $\phi$ such that
the following diagram is commutative
$$
\begin{tikzcd}
U\times K\arrow{rr}{\phi}[swap]{\sim}\arrow{dr}[swap]{p_1}&&\arrow{dl}{\nu_N}\nu_N^{-1}(U)\\
&U&
\end{tikzcd}
$$
and $\phi(u,k_1k_2)=k_1\phi(u,k_2)$. We set $\mathbf k=p_2\phi^{-1}$. It is a continuous
map from $\nu_N^{-1}(U)$ to $K$ such that $\phi(\nu_N(e),\mathbf k(e))=e$ for any $e\in\nu_N^{-1}(U)$.
It satisfies the following property:
$
\mathbf k(ke)=k\mathbf k(e)
$ 
for any $e\in\nu_N^{-1}(U)$ and $k\in K$.

We also consider the section $t:U\to E^N$ of $\nu_N$ given by
$t(u)=\phi(u,1)$. By definition, we get $\mathbf k(t(u))=1$.

Let $z:U(N)\to E^N$ be the map defined by $z(g)=t(\bar e)g$.
As $U(N)$ acts transitively on $E^N$, we get that $z$ is a fibre bundle.
Therefore, there exists an open neighbourhood $V$ of $t(\bar e)$
and a continuous section $y:V\to U(N)$ of $z$.
Hence for any $w\in V$, we get
\begin{equation}\label{eq:15}
w=z(y(w))=t(\bar e)y(w).
\end{equation}

Consider the set $U'=t^{-1}(V)$.
It is an open subset of $U$ and hence is an open neighbourhood of $\bar e$ in $E^N/K$.
Considering the restriction $\phi'=\phi|_{U'\times K}$, 
we get the following commutative diagram
$$
\begin{tikzcd}
U'\times K\arrow{rr}{\phi'}[swap]{\sim}\arrow{dr}[swap]{\pr_1}&&\arrow{dl}{\nu_N}\nu_N^{-1}(U')\\
&U'&
\end{tikzcd}
$$
We define the map $\psi:U'\times(\eta\kappa)^{-1}(\bar e)\to M^N$ by
$$
\begin{tikzcd}
\Big(u,\big(K([d],t(\bar e)),a_1,\ldots,a_\n\big)\Big)\arrow[mapsto]{r}{\psi}&
\Big(K([d],t(u)),a_1y(t(u)),\ldots,a_\n y(t(u))\Big).
\end{tikzcd}
$$
Suppose that the right-hand side of the above formula belongs to $\im\mu^N$.
Let us write $a_m=K(b_m,e_m)$
for some $b_m\in\BS_c(s_{f^m},v_{f^m})$ and $e_m\in E^N$.
By Lemma~\ref{lemma:10}, we get that
\begin{multline*}
Ke_my(t(u))=
\sigma^N_m\Big(K([d],t(u))\Big)
=K(d_{\min I}\cdots d_{f^1_1-1}d_{f^1_2+1}\cdots d_{f^2_1-1}\cdots\\
\shoveright{
\cdots d_{f^{m-1}_2+1}\cdots d_{f^m_1-1}\dot v_{f^1}\cdots\dot v_{f^{m-1}})^{-1}t(u).
}\\
\end{multline*}
\vspace{-29pt}

\noindent
for any $m=1,\ldots,\n$. By~(\ref{eq:15}), this equality is equivalent to
\begin{multline*}
Ke_m=K(d_{\min I}\cdots d_{f^1_1-1}d_{f^1_2+1}\cdots d_{f^2_1-1}\cdots\\
\shoveright{
\cdots d_{f^{m-1}_2+1}\cdots d_{f^m_1-1}\dot v_{f^1}\cdots\dot v_{f^{m-1}})^{-1}t(u)y(t(u))^{-1}
}\\
\shoveleft{
=K(d_{\min I}\cdots d_{f^1_1-1}d_{f^1_2+1}\cdots d_{f^2_1-1}
\cdots d_{f^{m-1}_2+1}\cdots d_{f^m_1-1}\dot v_{f^1}\cdots\dot v_{f^{m-1}})^{-1}t(\bar e)}\\
=\sigma^N_m\Big(K([d],t(\bar e))\Big).
\end{multline*}

\noindent
Hence we get a contradiction $(K([d],t(\bar e)),a_1,\ldots,a_\n)\in\im\mu^N$.
Therefore, the image of $\psi$ is a subset of $M^N\setminus\im\mu^N$.
It is easy to see that $\psi$ is actually a map to $(\eta\kappa)^{-1}(U')$.

In order to prove that $\psi$ is a homeomorphism, we need to construct the inverse map $\xi:(\eta\kappa)^{-1}(U')\to U'\times(\eta\kappa)^{-1}(\bar e)$.
It is easy to see that it is given by
$$
\hspace{-50pt}
\begin{tikzcd}
\Big(K([d],e),a_1,\ldots,a_\n\Big)\arrow[mapsto]{r}&\Big(\nu_N(e),\big(K(\mathbf k(e)^{-1}[d],t(\bar e)),a_1y(t(\nu_N(e)))^{-1},\ldots
\end{tikzcd}
$$
$$
\hspace{320pt}
\ldots ,a_\n y(t(\nu_N(e)))^{-1}\big)\Big).
$$
We get the following commutative diagram:
$$
\begin{tikzcd}
U'\times(\eta\varkappa)^{-1}(\bar e)\arrow{rr}{\psi}[swap]{\sim}\arrow{dr}[swap]{\pr_1}&&\arrow{dl}{\eta\varkappa}(\eta\varkappa)^{-1}(U')\\
&U'&
\end{tikzcd}
$$
Finally, note that $(\eta\varkappa)^{-1}(\bar e)$ are homeomorphic for different $\bar e$, as $E^N/K$ is connected and compact.
\end{proof}

\begin{lemma}\label{lemma:13}
Suppose that $(s,v)$ is of gallery type.
\begin{enumerate}
\item\label{lemma:13:p:1} If $\r<N<\infty$, then $H_c^n(M^N\setminus\im\mu^N,\k)=0$ for odd $n<2(N-\r)-\gld\k$.\\[-6pt]
\item\label{lemma:13:p:2} If $\r<N<\infty$, then the restriction map $H^n(M^N,\k)\to H^n(\BS_c(s,v)\times_K E^N,\k)$
                          induced by $\mu^N$ is surjective for $n<2(N-\r)-\gld\k-1$.\\[-6pt]
\item\label{lemma:13:p:3} For any $n$, the restriction map $H^n(M,\k)\to H^n_K(\BS_c(s,v),\k)$ induced by $\mu$
                          is surjective.
\end{enumerate}
\end{lemma}
\begin{proof}
We write the Leray spectral sequence for cohomologies with compact support for $\eta\varkappa$, where $\eta$ and $\varkappa$ are
as in Lemma~\ref{lemma:11}. As by part~\ref{lemma:11:p:2} of this lemma $\eta\varkappa$ is a fibre bundle and $E^N/K$ is
simply connected, it has the following second page:
$$
E_2^{p,q}=H_c^p\big(E^N/K,H_c^q((\eta\varkappa)^{-1}(\bar e),\k)\big),
$$
where $\bar e$ is an arbitrary point of $E^N/K$.
We claim that
{\renewcommand{\labelenumi}{{$(*)$}}
\renewcommand{\theenumi}{{$(*)$}}
\begin{enumerate}
\item\label{lemma:13:p:a} $H_c^q((\eta\varkappa)^{-1}(\bar e),\k)$ is free of finite rank and equals zero for odd $q$ if $q<2(N-\r)-\gld\k$;\\[-9pt]
\end{enumerate}}
In order to prove it, we consider the map $\widetilde\varkappa:(\eta\varkappa)^{-1}(\bar e)\to\eta^{-1}(\bar e)$
that is the restriction of $\varkappa$. Note that $\widetilde\varkappa$ is a fibre bundle
as a restriction of $\kappa$, which is a fibre bundle by Lemma~\ref{lemma:11}\ref{lemma:11:p:1}.
Hence we can consider the Leray spectral sequence for cohomologies with compact support
for $\widetilde\varkappa$. It has the following second page:
$$
\widetilde E_2^{r,t}=H_c^r\Big(\eta^{-1}(\bar e),\F^t\Big),
$$
where $\F^t=\H^t\widetilde\kappa_!\csh{\k}{(\eta\varkappa)^{-1}(\bar e)}$.
This is a local (not necessarily constant) system
on $\eta^{-1}(\bar e)\cong\BS_c(s^F,v^F)$. We claim that
{\renewcommand{\labelenumi}{{\rm ($**$)}}
\renewcommand{\theenumi}{{($**$)}}
\begin{enumerate}
\item\label{lemma:13:p:i} the stalks of $\F^t$ are free of finite rank and equal zero for odd $t$ if $t<2(N-\r)-\gld\k$;\\[-9pt]
\end{enumerate}}
To prove it, let us take an arbitrary point $b\in\eta^{-1}(\bar e)$ and consider the following Cartesian diagram:
$$
\begin{tikzcd}
\kappa^{-1}(b)\arrow[hook]{r}{\tilde\imath}\arrow{d}[swap]{\kappa_b}&(\eta\kappa)^{-1}(\bar e)\arrow{d}{\widetilde\kappa}\\
\{b\}\arrow[hook]{r}{i}&\eta^{-1}(\bar e)
\end{tikzcd}
$$
Hence by the proper base change, we get
$$
i^*\F^t=\H^ti^*\widetilde\kappa_!\csh{\k}{(\eta\varkappa)^{-1}(\bar e)}
=\H^t(\kappa_b)_!\tilde\imath^*\csh{\k}{(\eta\varkappa)^{-1}(\bar e)}=
\H^t(\kappa_b)_!\csh{\k}{\kappa^{-1}(b)}=H_c^t(\kappa^{-1}(b),\k).
$$
To compute the last module, consider the canonical projections $\eta_m:\BS_c(s_{f^m},v_{f^m})\times_K E^N\to E^N/K$
for $m=1,\ldots,\n$. Their direct product
$$
\eta_1\times\cdots\times\eta_\n:(\BS_c(s_{f^1},v_{f^1})\times_K E^N)\times\cdots\times(\BS_c(s_{f^\n},v_{f^\n})\times_KE^N)\to(E^N/K)^\n
$$
is a fibre bundle with fibre $\BS_c(s_{f^1},v_{f^1})\times\cdots\times\BS_c(s_{f^\n},v_{f^\n})$.
It has an affine paving by Proposition~\ref{proposition:bt10:1}. 
By Lemma~\ref{lemma:10} (the claim about the image), we get that $\kappa^{-1}(b)$ is homeomorphic to
$$
(\BS_c(s_{f^1},v_{f^1})\times_KE^N)\times\cdots\times(\BS_c(s_{f^\n},v_{f^\n})\times_KE^N)\setminus(\eta_1\times\cdots\times\eta_\n)^{-1}(\sigma^N_1(b),\ldots,\sigma^N_\n(b)),
$$
By~\cite[Lemma~19(2)]{NFBBSV} and the K\"unneth formula in form~(\ref{eq:coh:12}),
we get that $H^t((E^N/K)^\n,\k)$ is free of finite rank and
equals zero for odd $t$ if $t<2(N-\r)+1-\gld\k$. Therefore, we can apply~\cite[Lemma~23]{NFBBSV}
to 
the fibre bundle $\eta_1\times\cdots\times\eta_\n$.
We thus get that $H_c^t(\kappa^{-1}(b),\k)$ is free of finite rank and equals zero for odd $t$ if $t<2(N-\r)-\gld\k$.
Thus we have proved~\ref{lemma:13:p:i}.

By
Proposition~\ref{proposition:bt10:1},
the space $\BS_c(s^F,v^F)$ has an affine paving.
Therefore, by~\ref{lemma:13:p:i},
we get that $\widetilde E_2^{r,t}$ is zero
except the following cases both $r$ and $t$ are even; $t\ge 2(N-\r)-\gld\k$.
Moreover $\widetilde E_2^{r,t}$ is free of finite rank for $t<2(N-\r)-\gld\k$.

The differentials coming to and starting from $\widetilde E_a^{r,t}$ are zero for $a\ge2$
if $r+t<2(N-\r)-\gld\k$. Thus $\widetilde E_\infty^{r,t}=\widetilde E_2^{r,t}$ for $r+t<2(N-\r)-\gld\k$.
Therefore, $H_c^q((\eta\varkappa)^{-1}(\bar e),\k)$ is free of finite rank and equals zero for
odd $q$ if $q<2(N-\r)-\gld\k$. Thus~\ref{lemma:13:p:a} is proved.

Finally, let us look at $E_2^{p,q}$. By~\ref{lemma:13:p:a} and~\cite[Lemma~19(2)]{NFBBSV},
we get that $E_2^{p,q}$ is zero except the following cases:
$p$ and $q$ are both even; $p\ge2(N-\r)+1$; $q\ge2(N-\r)-\gld\k$.

That the differentials coming to and starting from $E_a^{p,q}$ are zero for $a\ge2$
if $p+q<2(N-\r)-\gld\k$. Thus $E_\infty^{p,q}=E_2^{p,q}$ for $p+q<2(N-\r)-\gld\k$.
It follows now from the Leray spectral sequence that $H^n_c(M^N\setminus\im\mu^N,\k)=0$ for odd
$n<2(N-\r)-\gld\k$.

\ref{lemma:13:p:2}. Let $n<2(N-\r)-\gld\k-1$. If $n$ is even, then  by the first part of this lemma,
we get an exact sequence
$$
 H^n(M^N,\k)\to H^n(\BS_c(s,v)\times_K E^N,\k)\to H_c^{n+1}(M^N\setminus\im\mu^N,\k)=0.
$$
If $n$ is odd, then the restriction under consideration is surjective as $H^n(\BS_c(s,v)\times_K E^N,\k)=0$
by~\cite[Lemma~20]{NFBBSV}.

\ref{lemma:13:p:3}. This result follows from the previous part and~\cite[Lemma~18]{NFBBSV}
and the following commutative diagram:
$$
\begin{tikzcd}
H^n(M^N,\k)\arrow{r}{(\mu^N)^*}\arrow[equal]{d}[swap]{\wr}&H^n(\BS_c(s)\times_KE^N,\k)\arrow[equal]{d}{\wr}\\
H^n(M,\k)\arrow{r}{\mu^*}&H_K^n(\BS_c(s),\k)
\end{tikzcd}
$$
which holds for $n<2(N-\r)+1$.
\end{proof}

{\bf Remark.}
It easily follows from this proof that the kernels of the restrictions $H^n(M,\k)\to H_K^n(\BS_c(s,v),\k)$
are free of finite rank. We do not need this fact in the sequel.

\subsection{Tensor product decomposition}\label{The_main_result}
Note that for a pair $(s,v)$ of gallery type, we get by the K\"unneth formula and~\cite[Lemma~18]{NFBBSV}
the map
$$
\mu^*:
H^\bullet(M,\k)\cong
H^\bullet_K(BS_c(s^F,v^F),\k)\otimes_\k
\bigotimes_{m=1}^{\n}{}_\k H^\bullet_K(\BS_c(s_{f^m},v_{f^m}),\k)
\to H^\bullet_K(\BS_c(s,v),\k).
$$
By Lemma~\ref{lemma:13}\ref{lemma:13:p:3}, this map is surjective.
We would like to study its kernel.
Let $Q=S\otimes_\k\cdots\otimes_\k S$ be the tensor product of $\n$ copies
of $S$. Clearly $Q$ is a commutative ring isomorphic to the polynomial ring $\k[x_1,\ldots,x_{\n}]$.
We consider $H^\bullet_K(BS_c(s^F,v^F),\k)$ as an $S$-$Q$-bimodule with the canonical left action of $S$
and the following right action of $Q$:
$$
h(c_1\otimes\cdots\otimes c_\n)=h\cup\sigma_1^*(c_1)\cup\cdots\cup\sigma_\n^*(c_\n)
$$
for $h\in H^\bullet_K(BS_c(s^F,v^F),\k)$ and $c_m\in S$.
Here $\sigma_m$ is the map as in Section~\ref{Embeddings_of_the_Borel_constructions}.
On the other hand, we consider the tensor product
$$
\bigotimes_{m=1}^{\n}{}_\k H^\bullet_K(\BS_c(s_{f^m},v_{f^m}),\k)
$$
as the left $Q$-module such that the $m$th copy of $S$ acts on $H^\bullet_K(\BS_c(s_{f^m},v_{f^m}),\k)$
canonically:
$$
(c_1\otimes\cdots\otimes c_\n)(h_1\otimes\cdots\otimes h_\n)=(\eta_1^*(c_1)\cup h_1)\otimes\cdots\otimes(\eta_\n^*(c_\n)\cup h_\n),
$$
where $\eta_m:\BS_c(s_{f^m},v_{f^m})\times_KE^\infty\to E^\infty/K$ is the canonical projection.

As the space $E^\infty/K$ is connected, we get $\sigma_m^*(c)=(\eta^F)^*(c)=c$, for any $c\in S^0=H^0(\pt,\k)$,
where $\eta^F:\BS_c(s^F,v^F)\times_KE^\infty\to E^\infty/K$ is the canonical projection.
Thus considering $Q$ as a $Q$-$\k$-bimodule, we get that
\begin{equation}\label{eq:bt6:10}
H_K^\bullet(BS_c(s^F,v^F),\k)\otimes_QQ\cong H_K^\bullet(BS_c(s^F,v^F),\k)
\end{equation}
as a $\k$-$\k$-bimodule with respect to the canonical actions of $\k$ on both sides.

%
%

In the proof of Lemma~\ref{lemma:10}, it was shown that the image of $\mu^N$ satisfies~(\ref{eq:13}).
This argument applies equally well to the case $N=\infty$. Therefore, we get the following commutative diagram:
$$
\begin{tikzcd}
\BS_c(s,v)\times_KE^\infty\arrow{r}{p_1\mu}\arrow{d}[swap]{p_{m+1}\mu}&\BS_c(s^F,v^F)\times_KE^\infty\arrow{d}{\sigma_m}\\
\BS_c(s_{f^m},v_{f^m})\times_KE^\infty\arrow{r}{\eta_m}& E^\infty/K
\end{tikzcd}
$$

Suppose that we have homogeneous elements $h\in H^\bullet_K(\BS_c(s^F,v^F),\k)$ and
$c_m\in S$, $h_m\in H^\bullet_K(\BS_c(s_{f^m},v_{f^m}),\k)$ for $m=1,\ldots,\n$.
Taking cohomologies in the above diagram and using identification~(\ref{eq:coh:12}),
we get
\begin{multline*}
\mu^*\big(h\otimes((c_1\otimes\cdots\otimes c_\n)(h_1\otimes\cdots\otimes h_\n))\big)=\mu^*\big(h\otimes(\eta_1^*(c_1)\cup h_1)\otimes\cdots\otimes(\eta_\n^*(c_\n)\cup h_\n)\big)\\
\shoveleft{
=\mu^*\big(p_1^*(h)\cup p_2^*(\eta_1^*(c_1)\cup h_1)\cup\cdots\cup p_{\n+1}^*(\eta_\n^*(c_\n)\cup h_\n)\big)}\\
=(p_1\mu)^*(h)\cup(\eta_1p_2\mu)^*(c_1)\cup(p_2\mu)^*(h_1)\cup\cdots\cup(\eta_\n p_{\n+1}\mu)^*(c_\n)\cup(p_{\n+1}\mu)^*(h_\n)\\
=(p_1\mu)^*(h)\cup(\sigma_1p_1\mu)^*(c_1)\cup(p_2\mu)^*(h_1)\cup\cdots\cup(\sigma_\n p_1\mu)^*(c_\n)\cup(p_{\n+1}\mu)^*(h_\n)\\
=\mu^*\big(p_1^*(h\cup\sigma_1^*(c_1)\cup\cdots\cup\sigma_\n^*(c_\n))\cup p_2^*(h_1)\cup\cdots\cup p_{\n+1}^*(h_\n)\big)\\
=\mu^*(h(c_1\otimes\cdots\otimes c_\n)\otimes h_1\otimes\cdots\otimes h_\n).
\end{multline*}
Hence
$$
h\otimes((c_1\otimes\cdots\otimes c_\n)(h_1\otimes\cdots\otimes h_\n))-h(c_1\otimes\cdots\otimes c_\n)\otimes h_1\otimes\cdots\otimes h_\n\in\ker\mu^*.
$$
Let $A$ be the $\k$-submodule of $H^\bullet(M,\k)$ generated by all the above differences.
Clearly $A\subset\ker\mu^*$. 
Therefore, applying~(\ref{eq:bt6:10}), we get that the quotient $H^\bullet(M,\k)/A$ is the following tensor product:
\begin{equation}\label{eq:bt6:11}
\begin{array}{l}\!\!\!\!\!
\displaystyle
H^\bullet_K(BS_c(s^F,v^F),\k)\otimes_Q\bigotimes_{m=1}^\n{\vphantom{\bigotimes}}_\k H^\bullet_K(\BS_c(s_{f^m},v_{f^m}),\k)\\
\qquad\qquad
\displaystyle
\cong H_K^\bullet(BS_c(s^F,v^F),\k)\otimes_Q\bigotimes_{m=1}^\n{\vphantom{\bigotimes}}_\k S\otimes_\k H^\bullet(\BS_c(s_{f^m},v_{f^m}),\k)\\
\qquad\qquad
\displaystyle
\cong H_K^\bullet(BS_c(s^F,v^F),\k)\otimes_QQ\otimes_\k\bigotimes_{m=1}^\n{\vphantom{\bigotimes}}_\k H^\bullet(\BS_c(s_{f^m},v_{f^m}),\k)\\
\qquad\qquad
\displaystyle
\cong H_K^\bullet(BS_c(s^F,v^F),\k)\otimes_\k\bigotimes_{m=1}^\n{\vphantom{\bigotimes}}_\k H^\bullet(\BS_c(s_{f^m},v_{f^m}),\k)\\
\qquad\qquad\qquad\qquad\qquad
\displaystyle
\cong S\otimes_\k H^\bullet(BS_c(s^F,v^F),\k)\otimes_\k\bigotimes_{m=1}^\n{\vphantom{\bigotimes}}_\k H^\bullet(\BS_c(s_{f^m},v_{f^m}),\k).
\end{array}
\end{equation}
On the other hand, by~\cite[Theorem~7]{NFBBSV}, applying the Leray spectral sequence, we get

$$
H^\bullet(\BS_c(s,v),\k)\cong H^\bullet(\BS_c(s^F,v^F),\k)\otimes_\k\bigotimes_{m=1}^\n{\vphantom{\bigotimes}}_\k H^\bullet(\BS_c(s_{f^m},v_{f^m}),\k).
$$
as graded $\k$-modules. Hence $H_K^\bullet(\BS_c(s,v),\k)\cong S\otimes_\k H^\bullet(\BS_c(s,v),\k)\cong H^\bullet(M,\k)/A$.
Hence we get $A=\ker\mu^*$ by the following result.

\begin{proposition} Let $L\subset L'\subset M$ be $\k$-modules such that
$$
M/L\cong \k^n,\quad M/L'\cong\k^n
$$
for some integer $n$. Then $L=L'$.
\end{proposition}
\begin{proof} It follows from~\cite[Theorem~3.6]{Rotman} applied to the quotient homomorphism
$$
M/L\to(M/L)/(L'/L)\cong M/L'.
$$
\end{proof}

\begin{theorem}\label{theorem:bt7:3} For a pair $(s,v)$ is of gallery type,
the embedding $\mu$ induces an isomorphism of left $S$-modules
$$
\mu^\circ:
H^\bullet_K(BS_c(s^F,v^F),\k)\otimes_Q\bigotimes_{m=1}^\n{\vphantom{\bigotimes}}_\k H^\bullet_K(\BS_c(s_{f^m},v_{f^m}),\k)
\ito H^\bullet_K(BS_c(s,v),\k).
$$
\end{theorem}
\begin{proof}
It remains to prove that this map is an isomorphism of $S$-modules.
It suffices to prove that $\mu^*$ is a morphism of $S$-modules.
We have the following commutative diagram:
$$
\begin{tikzcd}
\BS_c(s,v)\times_KE^\infty\arrow{r}{\eta}\arrow{d}[swap]{p_1\mu}&E^\infty/K\\
\BS_c(s^F,v^F)\times_KE^\infty\arrow{ur}[swap]{\eta_F}
\end{tikzcd}
$$
where $\eta$ is the canonical projection.

Suppose that we have homogeneous elements $h\in H^\bullet_K(BS_c(s^F,v^F)$, $c\in S$ and
$h_m\in H^\bullet_K(\BS_c(s_{f^m},v_{f^m}),\k)$ for $m=1,\ldots,\n$. Then we get
\begin{multline*}
\mu^*\big(c(h\otimes h_1\otimes\cdots\otimes h_\n)\big)=\mu^*\big(ch\otimes h_1\otimes\cdots h_\n\big)
=\mu^*\big(p_1^*(ch)\cup p_2^*(h_1)\cup\cdots\cup p_\n^*(h_\n)\big)\\
=\mu^*\big((\eta^Fp_1)^*(c)\cup p_1^*(h)\cup p_2^*(h_1)\cup\cdots\cup p_\n^*(h_\n)\big)
=(\eta^Fp_1\mu)^*(c)\cup\mu^*\big(p_1^*(h)\cup p_2^*(h_1)\cup\cdots\cup p_\n^*(h_\n)\big)\\
=\eta^*(c)\cup\mu^*\big(p_1^*(h)\cup p_2^*(h_1)\cup\cdots\cup p_\n^*(h_\n)\big)
=c\,\mu^*(h\otimes h_1\otimes\cdots\otimes h_\n).
\end{multline*}
\end{proof}

\section{Results for the big torus}\label{Results for the big torus}

\subsection{Fixed points} In Section~\ref{Definition via the compact torus}, we identified the set of $K$-fixed points $\BS_c(s)^K$ with
the set of generalized combinatorial galleries $\Gamma(s)$. We would like to know what the constructions of
Section~\ref{Nested_fibre_bundles} look like when restricted to the fixed points under this identification.
We set
$$
\Gamma(s,v)=\{\gamma\in\Gamma(s)\suchthat\forall r\in R:\gamma_{r_1}\gamma_{r_1+1}\cdots\gamma_{r_2}=v_r\}.
$$
We also define the map $p^F_K:\Gamma(s,v)\to\Gamma(s^F,v^F)$ by
$$
p^F_K(\gamma)_i=v^i\gamma_i(v^i)^{-1}.
$$
We get the following obvious result.

\begin{lemma} We have $\BS_c(s,v)^K=\Gamma(s,v)$. Moreover, the projection $p^F$ restricted to the $K$-fixed
points is $p^F_K$.
\end{lemma}

Note that $\mu$ maps $\Gamma(s,v)\times_K E^\infty$ to the following subspace of $M$:
$$
M_K=(\Gamma(s^F,v^F)\times_K E^\infty)\times\prod_{m=1}^\n\Gamma(s_{f^m},v_{f^m})\times_K E^\infty.
$$
We denote by $\mu_K:\Gamma(s,v)\times_K E^\infty\to M_K$ the restriction of $\mu$.
Thus we get the map
$$
\mu_K^*:H^\bullet(M_K,\k)\cong H^\bullet_K(\Gamma(s^F,v^F),\k)\otimes_\k\bigotimes_{m=1}^\n{\vphantom{\bigotimes}}_\k H^\bullet_K(\Gamma(s_{f^m},v_{f^m}),\k)\to H^\bullet_K(\Gamma(s,v),\k)
$$
similarly to Section~\ref{The_main_result}. This map is automatically surjective as $\mu_K$ is an embedding
and all spaces are discrete.
We also make $H^\bullet_K(\Gamma(s^F,v^F),\k)$ into an $S$-$Q$-bimodule
by letting $S$ act on the left canonically and defining the right $Q$-action by
$$
h(c_1\otimes\cdots\otimes c_\n)=h\cup(\sigma^K_1)^*(c_1)\cup\cdots\cup(\sigma^K_\n)^*(c_\n),
$$
where $\sigma^K_m:\Gamma(s^F,v^F)\times_K E^\infty\to E^\infty/K$ is the restriction of $\sigma_m$.
Repeating the arguments of Section~\ref{The_main_result}, we get the following result.

\begin{lemma} The embedding $\mu_K$ induces the isomorphism of left $S$-modules
$$
\mu^\circ_K:H^\bullet_K(\Gamma(s^F,v^F),\k)\otimes_Q\bigotimes_{m=1}^\n{\vphantom{\bigotimes}}_\k H^\bullet_K(\Gamma(s_{f^m},v_{f^m}),\k)\ito H^\bullet_K(\Gamma(s,v),\k).
$$
If $(s,v)$ is of gallery type, then the following diagram is commutative:
\begin{equation}\label{eq:bt6:9}
\begin{tikzcd}
\displaystyle
H^\bullet_K(\BS_c(s^F,v^F),\k)\otimes_Q\bigotimes_{m=1}^\n{\vphantom{\bigotimes}}_\k H^\bullet_K(\BS_c(s_{f^m},v_{f^m}),\k)\arrow{d}\arrow{r}[swap]{\sim}{\mu^\circ}&H^\bullet_K(\BS_c(s,v),\k)\arrow{d}\\
\displaystyle
H^\bullet_K(\Gamma(s^F,v^F),\k)\otimes_Q\bigotimes_{m=1}^\n{\vphantom{\bigotimes}}_\k H^\bullet_K(\Gamma(s_{f^m},v_{f^m}),\k)\arrow{r}[swap]{\sim}{\mu_K^\circ}&H^\bullet_K(\Gamma(s,v),\k)
\end{tikzcd}
\end{equation}
where the vertical arrows are restriction and tensor product of restrictions.
\end{lemma}

Let us compute the map $(\sigma_m^K)^*:S\to H_K^\bullet(\Gamma(s^F,v^F),\k)=S(\Gamma(s^F,v^F))$ exactly.
Note that the map $\sigma_m^K$ is actually a map
from the disjoint union $\Gamma(s^F,v^F)\times_K E^\infty\cong\Gamma(s^F,v^F)\times(E^\infty/K)$ of several copies
$E^\infty/K$ to $E^\infty/K$. For each $\gamma\in\Gamma(s^F,v^F)$, the restriction of this map
to $\{\gamma\}\times(E^\infty/K)$ equals
$\rho_{(\gamma^{f^m_1}v_{f^1}\cdots v_{f^{m-1}})^{-1}}/K$ in the notation of
Section~\ref{Equivariant_cohomology_of_a_point}. Hence by~(\ref{eq:pt:2}), we get
$$
(\sigma_m^K)^*(c)(\gamma)=\gamma^{f^m_1}v_{f^1}\cdots v_{f^{m-1}}c.
$$
From this formula, we get the right action of $Q$ on $H_K^\bullet(\Gamma(s^F,v^F),\k)$:
\begin{equation}\label{eq:bt6:16}
h(c_1\otimes\cdots\otimes c_\n)(\gamma)=h(\gamma)\prod_{m=1}^\n\gamma^{f^m_1}v_{f^1}\cdots v_{f^{m-1}}c_m.
\end{equation}

Finally, let us compute the map $\mu_K^\circ$. 
Let $\gamma$ be an arbitrary point of $\Gamma(s,v)$.
Let us denote $\delta=p^F_K(\gamma)$ for brevity.
We have the following commutative diagram
$$
\begin{tikzcd}
E^\infty/K\arrow{r}{j_\gamma}\arrow{rd}[swap]{j_\delta}&\Gamma(s,v)\times_K E^\infty\arrow{d}{p_1\mu_K}\\
&\Gamma(s^F,v^F)\times_K E^\infty
\end{tikzcd}
$$
Here $j_\gamma$ and $j_\delta$ are the maps defined in Section~\ref{Equivariant_cohomology_of_a_point}.
Moreover, for each $m=1,\ldots,\n$, we denote $\delta^m=\gamma_{f^m}$ and $w_m=\gamma^{f^m_1-1}$ for brevity.
We get the following commutative diagram:
$$
\begin{tikzcd}
E^\infty/K\arrow{r}{j_\gamma}\arrow{d}[swap]{\rho_{w_m^{-1}}/K}&\Gamma(s,v)\times_K E^\infty\arrow{d}{p_{m+1}\mu_K}\\
E^\infty/K\arrow{r}[swap]{j_{\delta^m}}&\Gamma(s_{f^m},v_{f^m})\times_K E^\infty
\end{tikzcd}
$$
Suppose that we have $h\in H_K^\bullet(\BS_c(s^F,v^F),\k)$ and $h_m\in H^\bullet_K(\Gamma(s_{f^m},v_{f^m}),\k)$
for $m=1,\ldots,\n$. Using the commutative diagrams above, we get
\begin{multline*}
\big(\mu_K^\circ(h\otimes h_1\otimes\cdots\otimes h_\n)\big)(\gamma)=j_\gamma^*\mu_K^*(p_1^*(h)\cup p_2^*(h_1)\cup\cdots\cup p_{\n+1}^*(h_\n))\\
=(p_1\mu_Kj_\gamma)^*(h)\cup(p_2\mu_Kj_\gamma)^*(h_1)\cup\cdots\cup(p_{\n+1}\mu_Kj_\gamma)^*(h_\n)\\
=j_\delta^*(h)\cup\big(j_{\delta^{(1)}}(\rho_{w_1^{-1}}/K)\big)^*(h_1)\cup\cdots\cup\(j_{\delta^{(\n)}}(\rho_{w_\n^{-1}}/K)\)^*(h_\n)\\
=h(\delta)\cup\big(\rho_{w_1^{-1}}/K\big)^*\big(h_1(\delta^{(1)})\big)\cup\cdots\cup\(\rho_{w_\n^{-1}}/K\)^*\big(h_\n(\delta^{(\n)})\big).
\end{multline*}
Applying~(\ref{eq:pt:2}) and omitting the sign of the cup product, we get
\begin{equation}\label{eq:bt6:12}
\mu_K^\circ(h\otimes h_1\otimes\cdots\otimes h_\n)(\gamma)=
h\big(p^F_K(\gamma)\big)\big(\gamma^{f^1_1-1}h_1(\gamma_{f^1})\big)\cdots\big(\gamma^{f^\n_1-1}h_\n(\gamma_{f^\n})\big).
\end{equation}

\subsection{Images, localization and good rings} We denote by $\X_c(s)$ the image of the restriction
$$H^\bullet_K(\BS_c(s),\k)\to H^\bullet_K(\Gamma(s),\k).$$
Similarly, let $\X_c(s,x)$ and $\X_c(s,v)$, where $x\in W$ and $v:R\to W$ is a map as in Section~\ref{Nested_fibre_bundles},
denote the images of the restrictions
$$
H^\bullet_K(\BS_c(s,x),\k)\to H^\bullet_K(\Gamma(s,x),\k),\quad\; H^\bullet_K(\BS_c(s,v),\k)\to H^\bullet_K(\Gamma(s,v),\k)
$$
respectively.

Let recall the homeomorphism $d_w:\BS_c(s)\ito\BS_c(s^w)$ from Section~\ref{Definition via the compact torus}.
We get the homeomorphism $d_w\times_K\rho_w:\BS_c(s)\times_K E^\infty\to\BS_c(s^w)\times_KE^\infty$.
Hence we get the isomorphism $d_w^{\im}:\X_c(s^w)\to\X(s)$ such that the diagram
$$
\begin{tikzcd}[column sep=15ex]
H^\bullet_K(\BS_c(s^w),\k)\arrow[two heads]{d}\arrow{r}{(d_w\times_K\rho_w)^*}&H^\bullet_K(\BS_c(s),\k)\arrow[two heads]{d}\\
\X_c(s^w)\arrow{r}{d_w^{\im}}&\X(s)
\end{tikzcd}
$$
is commutative.
Similarly, considering the spaces $\BS_c(s,x)$ and $\BS(s^w,wxw^{-1})$,
we get the isomorphism $d_{w,x}^{\im}:\X_c(s^w,wxw^{-1})\to\X(s,x)$.
Applying~(\ref{eq:pt:2}), the reader can easily check that
\begin{equation}\label{eq:bt7:5}
d_w^{\im}(f)(\lm)=w^{-1}f(\lm^\omega),\qquad d_{w,x}^{\im}(f)(\lm)=w^{-1}f(\lm^\omega).
\end{equation}
In a similar way, we can consider the homeomorphism $D_\gamma:\BS_c(s)\to\BS_c(s^{(\gamma)})$
defined in Section~\ref{Nested_fibre_bundles}. It is $K$-equivariant. Therefore, it induces the
isomorphisms $D_\gamma^{\im}:\X_c(s^{(\gamma)})\to\X_c(s)$ and $D_{\gamma,x}^{\im}:\X(s^{(\gamma)},x(\gamma^{\max})^{-1})\to\X(s,x)$.
Applying~(\ref{eq:pt:2}), we get
\begin{equation}\label{eq:bt7:6}
D_\gamma^{\im}(f)(\lm)=f(\gamma^{-1}\circ\lm),\qquad D_{\gamma,x}^{\im}(f)(\lm)=f(\gamma^{-1}\circ\lm),
\end{equation}
where we used the operations on the generalized combinatorial galleries defined in Section~\ref{Galleries}.

To obtain the tensor product theorem for images similar to Theorem~\ref{theorem:bt7:3},
we need to impose some restrictions on the ring of coefficients $\k$.

\begin{definition}\label{definition:2}
We say that a commutative ring $\k$ is good if it has finite global dimension, $2$ is invertible in $\k$ and
for any sequence of simple reflections $s$ and any $x\in W$, the restrictions
$$
H_T^\bullet(\BS(s),\k)\to H_K^\bullet(\Gamma(s),\k),\qquad H_T^\bullet(\BS(s,x),\k)\to H_K^\bullet(\Gamma(s,x),\k)
$$
are injective and the restriction
$
H_T^\bullet(\BS(s),\k)\to H_T^\bullet(\BS(s,x),\k)
$
is surjective.
\end{definition}
The results of~\cite{BTECBSV} imply that any principal ideal domain with invertible $2$
is a good ring.

Suppose that $(s,v)$ is of gallery type and $\k$ is a good ring. Then
we have the following commutative diagram:
\begin{equation}\label{eq:bt7:4}
\begin{tikzcd}
\displaystyle
H^\bullet_K(\BS_c(s^F,v^F),\k)\otimes_Q\bigotimes_{m=1}^\n{\vphantom{\bigotimes}}_\k H^\bullet_K(\BS_c(s_{f^m},v_{f^m}),\k)\arrow[two heads]{d}\arrow{r}[swap]{\sim}{\mu^\circ}&H^\bullet_K(\BS_c(s,v),\k)\arrow{d}{\wr}\\
\displaystyle
\X_c(s^F,v^F)\otimes_Q\bigotimes_{m=1}^\n{\vphantom{\bigotimes}}_\k\X_c(s_{f^m},v_{f^m})\arrow{d}\arrow{r}{\mu^{\im}_K}[swap]{\sim}&\X_c(s,v)\arrow[hook]{d}\\
\displaystyle
H^\bullet_K(\Gamma(s^F,v^F),\k)\otimes_Q\bigotimes_{m=1}^\n{\vphantom{\bigotimes}}_\k H^\bullet_K(\Gamma(s_{f^m},v_{f^m}),\k)\arrow{r}[swap]{\sim}{\mu_K^\circ}&H^\bullet_K(\Gamma(s,v),\k)
\end{tikzcd}
\end{equation}
This diagram can actually be obtained from Diagram~(\ref{eq:bt6:9}) by inserting the middle row.

\subsection{Surjectivity of restriction} We introduce the following parameter that we plan to use for induction.
Let $v:R\to W$ be a map, where $R$ is a nested structure on a finite totaly ordered set $I$
(see Section~\ref{Nested_structures}). Its {\it index} $\ind_Iv$ is a pair $(|R|,n)$, where $n=0$
if
$R$ is closed
and $n=1$ otherwise. These indices are compared lexicographically:
$(x,n)<(x',n')$ if and only if $x<x'$ or $x=x'$ and $n<n'$.

\begin{theorem}\label{theorem:2}
Suppose that $\k$ is a
good ring.
Then for any pair $(s,v)$ of gallery type the restriction
$H_K^\bullet(\BS_c(s),\k)\to H_K^\bullet(\BS_c(s,v),\k)$ is surjective.
\end{theorem}
\begin{proof} We denote by $\iota$ the natural embedding $\BS_c(s,v)\hookrightarrow \BS_c(s)$.
So we have to prove the surjectivity of $\iota^*:H_K^\bullet(\BS_c(s),\k)\to H_K^\bullet(\BS_c(s,v),\k)$.

We will prove the lemma by induction on $\ind_Iv$.
It is easy to see that two minimal possible values are $(0,1)$ and $(1,0)$.
In the first case, the restriction is an identical map. So we consider the second case.
We get $R=\{\span I\}$ and $s$ is itself of gallery type.
We have $\BS_c(s,v)=\BS_c(s,w)$, where $w=v_{\span R}$. Let $(x,t,\gamma)$ be a gallerification of $s$
(see Section~\ref{Galleries}). We get the following commutative diagram
$$
\begin{tikzcd}
\BS_c(s,w)\arrow[hook]{r}\arrow{d}{\wr}[swap]{\phi_w}&\BS_c(s)\arrow{d}[swap]{\wr}{\phi}\\
\BS(t,xwx^{-1}\gamma^{\max})\arrow[hook]{r}&\BS(t)
\end{tikzcd}
$$
(see the calculations from the proof of~\cite[Lemma~9]{NFBBSV}).
The isomorphisms $\phi$ and $\phi_w$ have the property $\phi(ka)=\dot x k\dot x^{-1}\phi(a)$
and $\phi_w(ka)=\dot x k\dot x^{-1}\phi_w(a)$ for $k\in K$ and $a\in\BS_c(s)$ or $a\in\BS_c(s,w)$
respectively. Multiplying by $E^\infty$ and taking taking quotients, we get the commutative diagram
$$
\begin{tikzcd}
\BS_c(s,w)\times_K E^\infty\arrow[hook]{r}\arrow{d}{\wr}[swap]{\phi_w\times_K\rho_x}&\BS_c(s)\times_KE^\infty\arrow{d}[swap]{\wr}{\phi\times_K\rho_x}\\
\BS(t,xwx^{-1}\gamma^{\max})\times_KE^\infty\arrow[hook]{r}&\BS(t)\times_KE^\infty
\end{tikzcd}
$$
Taking cohomologies, we get the commutative diagram
$$
\begin{tikzcd}
H_K(\BS_c(s,w),\k)&\arrow{l}H_K(\BS_c(s),\k)\\
H_T^\bullet(\BS(t,xwx^{-1}\gamma^{\max}),\k)\arrow{u}{\wr}&\arrow{l}H_T^\bullet(\BS(t),\k)\arrow{u}[swap]{\wr}
\end{tikzcd}
$$
The bottom arrow is surjective by the surjectivity condition. Hence the upper one is also surjective.

Now suppose that $\ind_I v$ is neither $(0,1)$ nor $(1,0)$. Let us consider the case
where the second component of $\ind_I v$ equals $1$, that is, $\span I\notin R$.
Let $F$ be the set of all maximal elements of $R$. Our assumption implies that $F\ne\emptyset$.
By Lemma~\ref{lemma:13}\ref{lemma:13:p:3} and the K\"unneth formula, we get that $H_K^\bullet(\BS_c(s,v),\k)$
is generated by elements $\mu^*p_1^*(h)$ and $\mu^*p_{m+1}^*(h_m)$,
where $h\in H^\bullet_K(\BS_c(s^F),\k)$ and $h_m\in H^\bullet_K(\BS_c(s_{f^m},v_{f^m}),\k)$.
We need to prove that all these elements can be lifted to some elements of $H^\bullet_K(\BS_c(s),\k)$.

First consider the element $\mu^*p_1^*(h)$. By~\cite[Lemma~25]{NFBBSV}, it suffices
to consider the case $h=\Sigma(s^F,i,1)^*(a)$ for some $i\in I^F\cup\{-\infty\}$ and $a\in S$.
We get
\begin{multline*}
\mu^*p_1^*(h)=(p_1\mu)^*(h)=(p^F\times_K\id)^*(h)\\
=(\Sigma(s^F,i,1)(p^F\times_K\id))^*(a)=(\Sigma(s,i,(v^i)^{-1})\iota)^*(a)
=\iota^*\Sigma(s,i,(v^i)^{-1})^*(a).
\end{multline*}
Hence $\Sigma(s,i,(v^i)^{-1})^*(a)$ delivers the required lifting of $\mu^*p_1^*(h)$.

Now let us lift elements $\mu^*p_{m+1}^*(h_m)=(p_{m+1}\mu)^*(h_m)=(q^{m,\infty})^*(h)$. To do it notice that we can extend the
map $q^{m,\infty}:\BS_c(s,v)\times_K E^\infty\to\BS_c(s_{f^m},v_{f^m})\times_K E^\infty$
defined in Section~\ref{Embeddings_of_the_Borel_constructions} to the map
$\tilde q^{m,\infty}:\BS_c(s)\times_K E^\infty\to\BS_c(s_{f^m})\times_K E^\infty$
given by~(\ref{eq:bt6:15}) as well as $q^{m,\infty}$. We get the commutative diagram
$$
\begin{tikzcd}
\BS_c(s)\times_K E^\infty\arrow{r}{\tilde q^{m,\infty}}&\BS_c(s_{f^m})\times_K E^\infty\\
\BS_c(s,v)\times_K E^\infty\arrow{r}{q^{m,\infty}}\arrow[hook]{u}{\iota\times\id}&\BS_c(s_{f^m},v_{f^m})\times_K E^\infty\arrow[hook]{u}
\end{tikzcd}
$$
Taking cohomologies, we get the commutative diagram
$$
\begin{tikzcd}
H^\bullet_K(\BS_c(s),\k)\arrow{d}[swap]{\iota^*}&\arrow{l}[swap]{(\tilde q^{m,\infty})^*}H^\bullet_K(\BS_c(s_{f^m}),\k)\arrow{d}\\
H^\bullet_K(\BS_c(s,v),\k)&\arrow{l}[swap]{(q^{m,\infty})^*}H^\bullet_K(\BS_c(s_{f^m},v_{f^m}),\k)
\end{tikzcd}
$$
Noting that $\ind_{[f^m]}v_{f^m}<\ind_Iv$, we get by the inductive hypothesis that there exists some
$\tilde h\in H^\bullet_K(\BS_c(s_{f^m}),\k)$ whose restriction to $\BS_c(s_{f^m},v_{f^m})$
is $h$. From the above diagram, we get $(q^{m,\infty})^*(h)=\iota^*(\tilde q^{m,\infty})^*(\tilde h)$.
Hence $(\tilde q^{m,\infty})^*(\tilde h)$ delivers the required lifting.

Now let us consider the case where the second component of $\ind_I v$ equals $0$,
that is, $\span I\in R$.
Let $F$ be the set of all maximal elements of the set $\widetilde R=R\setminus\{\span I\}$. As $\ind_I v\ne(1,0)$, we get that
$F\ne\emptyset$. The arguments of the previous case work now with the only exception: we need to show that
$\mu^*p_1^*(h)$ can be lifted to an element of $H^\bullet_K(\BS_c(s),\k)$
for any $h\in H^\bullet_K(\BS_c(s^F,v^F),\k)$. We denote $w=v_{\span I}$. Then we have $\BS_c(s^F,v^F)=\BS_c(s^F,w)$.
Consider
the restriction $\tilde v=v|_{\widetilde R}$.
Note that $\ind_I\tilde v<\ind_Iv$.
Let $\tilde p^F:\BS_c(s,\tilde v)\to\BS_c(s^F)$ be
the projection along $F$. We get the following commutative diagram:
$$
\begin{tikzcd}
{}                               &\arrow[hook]{ld}\BS_c(s,\tilde v)\arrow{r}{\tilde p^F}&\BS_c(s^F)\\
\BS_c(s)    &  & \\
                                &\arrow[hook]{lu}{\iota}\BS_c(s,v)\arrow[hook]{uu}\arrow{r}{p^F}&\BS_c(s^F,w)\arrow[hook]{uu}
\end{tikzcd}
$$
For cohomologies, we get the commutative diagram:
$$
\begin{tikzcd}
{}                               &H^\bullet_K(\BS_c(s,\tilde v),\k)\arrow{dd}&\arrow{l}[swap]{(\tilde p^F)^*}H^\bullet_K(\BS_c(s^F),\k)\arrow{dd}\\
H^\bullet_K(\BS_c(s),\k)\arrow{ru}\arrow{rd}[swap]{\iota^*}    &  & \\
                                &H^\bullet_K(\BS_c(s,v),\k)&\arrow{l}[swap]{(p^F)^*}H^\bullet_K(\BS_c(s^F,w),\k)
\end{tikzcd}
$$
The right vertical arrow corresponds to the special case of this theorem, where the index of $v$ is $(1,0)$,
which already was considered. Hence it is surjective and we can lift $h$ to an element $h'\in H^\bullet_K(\BS_c(s^F),\k)$.
By the inductive hypothesis, $(\tilde p^F)^*(h')$ can be lifted to an element $h''\in H^\bullet_K(\BS_c(s),\k)$.
The above diagram proves that $\iota^*(h'')=(p^F)^*(h)$.
\end{proof}

\subsection{Returning to the big torus} Even though the above constructions of $K$-invariant subspaces
do not have $T$-invariant counterparts, we can profit from them to obtain results for $T$-equivariant
cohomologies.

Suppose that $s$ is a sequence of simple reflections. We denote by $\X(s)$, $\X(s,x)$ and $\X(s,v)$ the images
of the restrictions $H_T^\bullet(\BS(s),\k)\to H_T^\bullet(\Gamma(s),\k)$, $H_T^\bullet(\BS(s),\k)\to H_T^\bullet(\Gamma(s,x),\k)$
and $H_T^\bullet(\BS(s),\k)\to H_T^\bullet(\Gamma(s,v),\k)$ respectively.
Our identification of $T$- and $K$-equivariant cohomologies and isomorphism of Sections~\ref{Isomorphism_of_two_constructions}
prove that $\X(s)=\X_c(s)$. If $\k$ is a good ring, then and $\X(s,x)=\X_c(s,x)$ and
$\X(s,v)=\X_c(s,v)$ by~Theorem~\ref{theorem:2}.

\begin{theorem}\label{theorem:bt7:1}
Let $s$ be a sequence of simple reflections on a finite totaly ordered set $I$,
$R$ be a nested structure on $I$ and $v:R\to W$ be an arbitrary function.
Let $F=\{f^1<\cdots<f^\n\}$ be the set of all maximal pairs of $R$. Suppose that $s^F$ and $(s,v)$
are of gallery type and that $\k$ is a good ring. Let $(t,\gamma,x)$ be a gallerification of $s^F$.
We consider $\X(t)$ as an $S$-$Q$-bimodule with the following actions (denoted by $\cdot$):
\begin{equation}\label{eq:bt7:2}
(c\cdot h)(\lambda)=(xc)h(\lambda),\quad
(h\cdot(c_1\otimes\cdots\otimes c_\n))(\lm)=h(\lm)\prod_{m=1}^\n\lm^{f^{m}_1}(\gamma^{f^{m}_1})^{-1}xv_{f^1}\cdots v_{f^{m-1}}c_m.
\end{equation}
Then there is the isomorphism of left $S$-modules
$$
\mu_T^{\im}:\X(t)\otimes_Q\bigotimes_{m=1}^\n\!{}_k\,\X(s_{f^m},v_{f^m})\ito\X(s,v)
$$
given by
\begin{equation}\label{eq:bt7:3}
\mu_T^{\im}(h\otimes h_1\otimes\cdots\otimes h_\n)(\lm)=\Big(x^{-1}h\big(\gamma\circ p^F_K(\lm)^x\big)\Big)\prod_{m=1}^\n\lm^{f^m_1-1}h_m(\lm_{f^m}).
\end{equation}
\end{theorem}
\begin{proof} The middle line of Diagram~(\ref{eq:bt7:4}) can be written as
$$
\mu^{\im}_K:\X_c(s^F)\otimes_Q\bigotimes_{m=1}^\n\!{}_k\,\X(s_{f^m},v_{f^m})\ito\X(s,v).
$$
It remains to find an isomorphism for the $S$-$Q$-bimodule $\X_c(s^F)$. We have $s^F=(t^{(\gamma)})^{x^{-1}}$.
So we have the following isomorphisms:
$$
\begin{tikzcd}
\X_c(s^F)\arrow{r}{d^{\im}_{x^{-1}}}[swap]{\sim}&\X_c(t^{(\gamma)})\arrow{r}{D_\gamma^{\im}}[swap]{\sim}&\X_c(t)=\X(t).
\end{tikzcd}
$$
Let $\phi=D_\gamma^{\im}d^{\im}_{x^{-1}}$ denote their composition. Applying direct formulas~(\ref{eq:bt7:5}) and~(\ref{eq:bt7:6}),
we get
$$
\phi(h)(\lm)=\Big(D_\gamma^{\im}d^{\im}_{x^{-1}}(h)\Big)(\lm)=d^{\im}_{x^{-1}}(h)(\gamma^{-1}\circ\lm)
=xh\Big((\gamma^{-1}\circ\lm)^{x^{-1}}\Big).
$$
We can easily invert $\phi$ as follows: $\phi^{-1}(h)(\lm)=x^{-1}h(\gamma\circ\lm^x)$.
For any $c\in S$ and $h\in\X(t)$,
we get
\begin{multline*}
c\cdot h(\lambda)=\phi(c\phi^{-1}(h))(\lambda)=x(c\phi^{-1}(h))\Big((\gamma^{-1}\circ\lm)^{x^{-1}}\Big)\\
=(xc)(x\phi^{-1}(h))\Big((\gamma^{-1}\circ\lm)^{x^{-1}}\Big)=(xc)h(\lambda).
\end{multline*}
On the other hand, for any $c_1,\ldots,c_\n\in S$, applying~(\ref{eq:bt6:16}), we get
\begin{multline*}
(h\cdot(c_1\otimes\cdots\otimes c_\n))(\lm)=\phi(\phi^{-1}(h)(c_1\otimes\cdots\otimes c_\n))(\lm)
=x(\phi^{-1}(h)(c_1\otimes\cdots\otimes c_\n))\Big((\gamma^{-1}\circ\lm)^{x^{-1}}\Big)\\
=x\(\phi^{-1}(h)\Big((\gamma^{-1}\circ\lm)^{x^{-1}}\Big)\prod_{m=1}^\n\Big((\gamma^{-1}\circ\lm)^{x^{-1}}\Big)^{f^m_1}v_{f^1}\cdots v_{f^{m-1}}c_m\)\\
=h(\lm)\prod_{m=1}^\n x\Big((\gamma^{-1}\circ\lm)^{x^{-1}}\Big)^{f^m_1}v_{f^1}\cdots v_{f^{m-1}}c_m\\
=h(\lm)\prod_{m=1}^\n\Big(\gamma^{-1}\circ\lm\Big)^{f^m_1}xv_{f^1}\cdots v_{f^{m-1}}c_m\\
=h(\lm)\prod_{m=1}^\n\lm^{f^m_1}(\gamma^{f^m_1})^{-1}xv_{f^1}\cdots v_{f^{m-1}}c_m.
\end{multline*}
From~(\ref{eq:bt6:12}), we get the exact formula
\begin{multline*}
\mu^{\im}_T(h\otimes h_1\otimes\cdots\otimes h_\n)(\lm)=
\mu_K^\circ(\phi^{-1}(h)\otimes h_1\otimes\cdots\otimes h_\n)(\lm)\\
=\phi^{-1}(h)\big(p^F_K(\lm)\big)\big(\lm^{f^1_1-1}h_1(\lm_{f^1})\big)\cdots\big(\lm^{f^\n_1-1}h_\n(\lm_{f^\n})\big)\\
=x^{-1}h\Big(\gamma\circ p^F_K(\lm)^x\Big)\big(\lm^{f^1_1-1}h_1(\lm_{f^1})\big)\cdots\big(\lm^{f^\n_1-1}h_\n(\lm_{f^\n})\big).
\end{multline*}
\end{proof}

\begin{theorem}\label{theorem:bt7:2}
Let $s$ be a sequence of simple reflections on a finite totaly ordered set $I$,
$R$ be a closed nested structure on $I$ and $v:R\to W$ be an arbitrary function.
Let $F=\{f^1<\cdots<f^\n\}$ be the set of all maximal pairs of $R\setminus\{\span I\}$. Suppose that $(s,v)$
is of gallery type and that $\k$ is a good ring. Let $(t,\gamma,x)$ be a gallerification of $s^{(\span I,v)}$
and $w=v_{span I}$.
Then there is the isomorphism of left $S$-modules
$$
\mu_T^{\im}:\X(t,xw(xv_{f^1}\cdots v_{f^\n})^{-1}\gamma^{\max})\otimes_Q\bigotimes_{m=1}^\n{}_k\X(s_{f^m},v_{f^m})\ito\X(s,v).
$$
The structure of an $S$-$Q$-module on the first factor is given by~(\ref{eq:bt7:2})
and the isomorphism is given by~(\ref{eq:bt7:3}).
\end{theorem}
\begin{proof}
We have the isomorphism of the middle line of~(\ref{eq:bt7:4})
$$
\mu^{\im}_K:\X_c(s^F,v^F)\otimes_Q\bigotimes_{m=1}^\n\!{}_k\,\X(s_{f^m},v_{f^m})\ito\X(s,v).
$$
Now note that $s^F=s^{(\span I,v)}$ and $v^F_{\span I}=w(v_{f^1}\cdots v_{f^\n})^{-1}$.
Hence
$$
\X_c(s^F,v^F)=\X_c(s^{(\span I,v)},w(v_{f^1}\cdots v_{f^\n})^{-1})\cong\X(t,xw(xv_{f^1}\cdots v_{f^\n})^{-1}\gamma^{\max}).
$$
To describe the structure of an $S$-$Q$-bimodule on the right set above and
to describe exactly the map $\mu_T^{\im}$, we just repeat the proof of Theorem~\ref{theorem:bt7:2}.
\end{proof}

\subsection{Dual modules}\label{Dual modules} Finally, we describe how the above results induce the maps between the dual modules.
In what follows, we assume that $\k$ is a good ring. For a sequence of simple reflections $s$ on $I$,
a nested structure $R$ on $I$ and $v:R\to W$,
we consider the following set of maps:
$$
D\X(s,v)=\{\alpha:\Gamma(s,v)\to\Frac(S)\suchthat(\alpha,g)\in S\;\forall g\in\X(s,v)\},
$$
where $\Frac(S)$ is the field of fractions of $S$ and $(,)$ is the standard Euclidian product:
$$
(\alpha,g)=\sum_{\lm\in\Gamma(s,v)}\alpha(\lm)g(\lm).
$$
Let $\tilde R$ be a subset of $R$ and $\tilde v=v|_{\tilde R}$. Then for any $\alpha\in D\X(s,v)$ the continuation by zero
$\tilde\alpha$ to $\Gamma(s,\tilde v)$ obviously belongs to $D\X(s,\tilde v)$. Indeed for any $g\in\X(s,\tilde v)$, we get
$$
(\tilde\alpha,g)=(\alpha,g|_{\Gamma(s,v)})\in S,
$$
as $g|_{\Gamma(s,v)}\in\X(s,v)$.

\begin{corollary}\label{corollary:1}
It the notation of Theorem~\ref{theorem:bt7:1}, there is the map
$$
D\mu_T^{\im}:D\X(t)\times\prod_{m=1}^\n D\X(s_{f^m},v_{f^m})\to D\X(s,v)
$$
given by
\begin{equation}\label{eq:bt7:9}
D\mu_T^{\im}(\alpha,\alpha_1,\ldots,\alpha_\n)(\lm)=\Big(x^{-1}\alpha\big(\gamma\circ p^F_K(\lm)^x\big)\Big)\prod_{m=1}^\n\lm^{f^m_1-1}\alpha_m(\lm_{f^m}).
\end{equation}
\end{corollary}
\begin{proof} Let us denote $\beta=D\mu_T^{\im}(\alpha,\alpha_1,\ldots,\alpha_\n)$ for brevity.
As $\beta$ is clearly a map from $\Gamma(s,v)$ to $\Frac(S)$, we only have to prove that $\beta\in D\X(s,v)$.
Applying the definition of $D\X(s,v)$ above, we see that we have to sum over $\lm\in\Gamma(s,v)$. It is easy to prove that the map
$\lm\mapsto(p_K^F(\lm),\lm_{f^1},\ldots,\lm_{f^\n})$ is a bijection from
$\Gamma(s,v)$ to $\Gamma(s^F)\times\Gamma(s_{f^1},v_{f^1})\times\cdots\times\Gamma(s_{f^\n},v_{f^\n})$.
Indeed the inverse map is given by the map $\theta$ from~\cite[Section~3.4]{NFBBSV}.
Hence the map $\lm\mapsto(\gamma\circ p_K^F(\lm)^x,\lm_{f^1},\ldots,\lm_{f^\n})$ is a bijection from
$\Gamma(s,v)$ to $\Gamma(t)\times\Gamma(s_{f^1},v_{f^1})\times\cdots\times\Gamma(s_{f^\n},v_{f^\n})$.
Therefore, we sum over $\delta=\gamma\circ p_K^F(\lm)^x$ and
the segments $\lm_{f^1},\ldots,\lm_{f^\n}$.
We get $p_K^F(\lambda)^{f^m_1-1}=\lm^{f^m_1-1}(v_1\cdots v_{m-1})^{-1}$ for any $\lm\in\Gamma(s,v)$.
It follows from this formula that
\begin{multline*}
\delta^{f^m_1-1}=(\gamma\circ p^F_K(\lm)^x)^{f^m_1-1}=(p^F_K(\lm)^x)^{f^m_1-1}\gamma^{f_1^m-1}
=xp^F_K(\lm)^{f^m_1-1}x^{-1}\gamma^{f_1^m-1}\\
=x \lm^{f^m_1-1}(v_1\cdots v_{m-1})^{-1} x^{-1}\gamma^{f_1^m-1}.
\end{multline*}
Hence
$$
\lm^{f^m_1-1}=x^{-1}\delta^{f^m_1-1}(\gamma^{f_1^m-1})^{-1}xv_1\cdots v_{m-1}.
$$
Now we compute the scalar product as follows:
\begin{equation}\label{eq:bt7:8}
\begin{array}{l}
(\beta,\mu_T^{\im}(h\otimes h_1\otimes\cdots\otimes h_\n))\\
\hspace{90pt}\displaystyle=\sum_{\delta\in\Gamma(t)}\big(x^{-1}(\alpha(\delta)h(\delta))\big)\prod_{m=1}^\n x^{-1}\delta^{f^m_1-1}(\gamma^{f_1^m-1})^{-1}xv_1\cdots v_{m-1}c_m,
\end{array}
\end{equation}
where $c_m=(\alpha_m,h_m)$, which is an element of $S$.

For each $m=1,\ldots,\n$, we consider the function $g_m:\Gamma(t)\to S$ defined by
$g_m(\delta)=\delta^{f^m_1-1}(\gamma^{f_1^m-1})^{-1}xv_1\cdots v_{m-1}c_m$. Note that $g_m$ belongs to $\X(t)$,
which we identify with $\X_c(t)$,
as this map is the restriction of $\Sigma(t,f_1^m-1,(\gamma^{f_1^m-1})^{-1}xv_1\cdots v_{m-1})^*(c_m)$ to $\Gamma(t)$.
From~(\ref{eq:bt7:8}), we get
$$
(\beta,\mu_T^{\im}(h\otimes h_1\otimes\cdots\otimes h_\n))=x^{-1}(\alpha,hg_1\cdots g_\n).
$$
The right-hand side belongs to $S$, as we take the (cup) product $hg_1\cdots g_\n$ of elements of
$\X(t)$.
\end{proof}

The same arguments allow us to prove the following result.

\begin{corollary}\label{corollary:2}
It the notation of Theorem~\ref{theorem:bt7:2}, there is the map
$$
D\X(t,xw(xv_{f^1}\cdots v_{f^\n})^{-1}\gamma^{\max})\times\prod_{m=1}^\n D\X(s_{f^m},v_{f^m})\to D\X(s,v)
$$
given by~(\ref{eq:bt7:9}).
\end{corollary}
\subsection{Example of a decomposition}\label{Example of a decomposition}
Suppose that $G=\SL_5(\C)$ with the following Dynkin diagram:

\begin{center}
\setlength{\unitlength}{1.2mm}
\begin{picture}(45,0)
\put(0,0){\circle{1}}
\put(10,0){\circle{1}}
\put(20,0){\circle{1}}
\put(30,0){\circle{1}}

\put(0.5,0){\line(1,0){9}}
\put(10.5,0){\line(1,0){9}}
\put(20.5,0){\line(1,0){9}}

\put(-1,-3.5){$\scriptstyle\alpha_1$}
\put(9,-3.5){$\scriptstyle\alpha_2$}
\put(19,-3.5){$\scriptstyle\alpha_3$}
\put(29,-3.5){$\scriptstyle\alpha_4$}
\end{picture}
\end{center}

\hspace{20mm}

\noindent
The simple reflections are $\omega_i=s_{\alpha_i}$. We consider the following sequence and the element of the Weyl group
$$
s=(\omega_4, \omega_3,  \omega_4, \omega_2,\omega_1,\omega_2,\omega_1,\omega_2, \omega_3,\omega_4),\quad
w=\omega_1\omega_3\omega_4\omega_3.
$$

We would like to study the module $\X(s,w)$ under the assumption that the ring of coefficients $\k$ is good.
It is easy to note that $\Gamma(s,w)=\Gamma(s,v)$, where $v$ is the map from $R=\{(1,10),(4,8)\}$ to $W$ given by
$v((1,10))=w$ and $v((4,8))=\omega_1$.
Thus $\X(s,w)=\X(s,v)$. We can apply Theorem~\ref{theorem:bt7:2} to $\X(s,v)$ for $F=\{(4,8)\}$.
Let $t=(\omega_4,\omega_3,\omega_4,\omega_3,\omega_4)$ and $\gamma=(1,1,1,1,1,1)$.
As we have $s^{(\span I,v)}=s^F=t$,
the triple $(1,t,\gamma)$ is a gallerification of $s^{(\span I,v)}$.
We obtain the isomorphism of left $S$-modules
$$
\mu_T^{\im}:\X(t,\omega_3\omega_4\omega_3)\otimes_S\X((\omega_2,\omega_1,\omega_2,\omega_1,\omega_2),\omega_1)\ito\X(s,w).
$$
Here all actions of $S$ are canonical except the right action on $\X(t,\omega_3\omega_4\omega_3)$.
It is given by $(h\cdot c)(\lm)=h(\lm)(\lm_1\lm_2\lm_3c)$, where $h\in\X(t,\omega_3\omega_4\omega_3)$ and $c\in S$.
The above isomorphism is given by the formula
$$
\mu_T^{\im}(h\otimes g)(\lm)=h(\lm_1,\lm_2,\lm_3,\lm_9,\lm_{10})(\lm_1\lm_2\lm_3g(\lm_4,\lm_5,\lm_6,\lm_7,\lm_8)).
$$
We can also represent the minimal degree element $a\in D\X(s,w)$ as a twisted product of the minimal degree elements
$b\in D\X(t,\omega_3\omega_4\omega_3)$ and $c\in D\X((\omega_2,\omega_1,\omega_2,\omega_1,\omega_2),\omega_1)$.
As all three Bott-Samelson varieties are smooth we get the following values of these functions:
$a$ maps
\begin{multline*}
(1,1,\omega_{{4}},1,1,1,\omega_{{1}},1,\omega_{{3}},\omega_{{4}}),
(\omega_{{4}},1,1,1,1,1,\omega_{{1}},1,\omega_{{3}},\omega_{{4}}),
(1,1,
\omega_{{4}},\omega_{{2}},1,\omega_{{2}},\omega_{{1}},1,\omega_{{3}},
\omega_{{4}}),\\
(\omega_{{4}},1,1,\omega_{{2}},1,\omega_{{2}},\omega_{{1
}},1,\omega_{{3}},\omega_{{4}}),(1,1,\omega_{{4}},1,\omega_{{1}},1,1,1
,\omega_{{3}},\omega_{{4}}),(\omega_{{4}},1,1,1,\omega_{{1}},1,1,1,
\omega_{{3}},\omega_{{4}}),\\
(1,1,\omega_{{4}},1,\omega_{{1}},\omega_{{2
}},1,\omega_{{2}},\omega_{{3}},\omega_{{4}}),(\omega_{{4}},1,1,1,
\omega_{{1}},\omega_{{2}},1,\omega_{{2}},\omega_{{3}},\omega_{{4}}),\\
(1
,1,\omega_{{4}},\omega_{{2}},\omega_{{1}},\omega_{{2}},\omega_{{1}},
\omega_{{2}},\omega_{{3}},\omega_{{4}}),(\omega_{{4}},1,1,\omega_{{2}}
,\omega_{{1}},\omega_{{2}},\omega_{{1}},\omega_{{2}},\omega_{{3}},
\omega_{{4}}),\\
(\omega_{{4}},\omega_{{3}},1,1,1,1,\omega_{{1}},1,1,
\omega_{{4}}),(\omega_{{4}},\omega_{{3}},1,\omega_{{2}},1,\omega_{{2}}
,\omega_{{1}},1,1,\omega_{{4}}),(\omega_{{4}},\omega_{{3}},1,1,\omega_
{{1}},1,1,1,1,\omega_{{4}}),\\
(\omega_{{4}},\omega_{{3}},1,1,\omega_{{1}
},\omega_{{2}},1,\omega_{{2}},1,\omega_{{4}}),
(\omega_{{4}},\omega_{{3
}},1,\omega_{{2}},\omega_{{1}},\omega_{{2}},\omega_{{1}},\omega_{{2}},
1,\omega_{{4}}),\\(1,\omega_{{3}},\omega_{{4}},1,1,1,\omega_{{1}},1,
\omega_{{3}},1),(1,\omega_{{3}},\omega_{{4}},\omega_{{2}},1,\omega_{{2
}},\omega_{{1}},1,\omega_{{3}},1),(1,\omega_{{3}},\omega_{{4}},1,
\omega_{{1}},1,1,1,\omega_{{3}},1),\\(1,\omega_{{3}},\omega_{{4}},1,
\omega_{{1}},\omega_{{2}},1,\omega_{{2}},\omega_{{3}},1),(1,\omega_{{3
}},\omega_{{4}},\omega_{{2}},\omega_{{1}},\omega_{{2}},\omega_{{1}},
\omega_{{2}},\omega_{{3}},1),\\(\omega_{{4}},\omega_{{3}},\omega_{{4}},1
,1,1,\omega_{{1}},1,1,1),(\omega_{{4}},\omega_{{3}},\omega_{{4}},
\omega_{{2}},1,\omega_{{2}},\omega_{{1}},1,1,1),(\omega_{{4}},\omega_{
{3}},\omega_{{4}},1,\omega_{{1}},1,1,1,1,1),\\(\omega_{{4}},\omega_{{3}}
,\omega_{{4}},1,\omega_{{1}},\omega_{{2}},1,\omega_{{2}},1,1),(\omega_
{{4}},\omega_{{3}},\omega_{{4}},\omega_{{2}},\omega_{{1}},\omega_{{2}}
,\omega_{{1}},\omega_{{2}},1,1)
\end{multline*}
to
\begin{multline*}
{\frac {1}{\alpha_{{2}}\alpha_{{3}}\alpha_{{4}}\alpha_{{1}}}},-{\frac {1}{\alpha_{{4}} \left( \alpha_{{3}}+\alpha_{{4}} \right) \alpha
_{{2}}\alpha_{{1}}}},-{\frac {1}{\alpha_{{3}}\alpha_{{4}}\alpha_{{2}}
 \left( \alpha_{{1}}+\alpha_{{2}} \right) }},{\frac {1}{ \left( \alpha
_{{3}}+\alpha_{{4}} \right) \alpha_{{4}}\alpha_{{2}} \left( \alpha_{{1
}}+\alpha_{{2}} \right) }},\\-{\frac {1}{\alpha_{{3}}\alpha_{{4}}
 \left( \alpha_{{1}}+\alpha_{{2}} \right) \alpha_{{1}}}},{\frac {1}{
 \left( \alpha_{{3}}+\alpha_{{4}} \right)  \left( \alpha_{{1}}+\alpha_
{{2}} \right) \alpha_{{4}}\alpha_{{1}}}},-{\frac {1}{\alpha_{{3}}
\alpha_{{4}}\alpha_{{2}} \left( \alpha_{{1}}+\alpha_{{2}} \right) }},\\{
\frac {1}{ \left( \alpha_{{3}}+\alpha_{{4}} \right) \alpha_{{4}}\alpha
_{{2}} \left( \alpha_{{1}}+\alpha_{{2}} \right) }},{\frac {1}{\alpha_{
{3}}\alpha_{{4}}\alpha_{{2}} \left( \alpha_{{1}}+\alpha_{{2}} \right)
}},-{\frac {1}{ \left( \alpha_{{3}}+\alpha_{{4}} \right) \alpha_{{4}}
\alpha_{{2}} \left( \alpha_{{1}}+\alpha_{{2}} \right) }},\\-{\frac {1}{
 \left( \alpha_{{2}}+\alpha_{{3}}+\alpha_{{4}} \right) \alpha_{{3}}
 \left( \alpha_{{3}}+\alpha_{{4}} \right) \alpha_{{1}}}},{\frac {1}{
\alpha_{{3}} \left( \alpha_{{3}}+\alpha_{{4}} \right)  \left( \alpha_{
{2}}+\alpha_{{3}}+\alpha_{{4}} \right)  \left( \alpha_{{1}}+\alpha_{{2
}}+\alpha_{{3}}+\alpha_{{4}} \right) }},\\{\frac {1}{\alpha_{{3}}
 \left( \alpha_{{1}}+\alpha_{{2}}+\alpha_{{3}}+\alpha_{{4}} \right)
 \left( \alpha_{{3}}+\alpha_{{4}} \right) \alpha_{{1}}}},{\frac {1}{
\alpha_{{3}} \left( \alpha_{{3}}+\alpha_{{4}} \right)  \left( \alpha_{
{2}}+\alpha_{{3}}+\alpha_{{4}} \right)  \left( \alpha_{{1}}+\alpha_{{2
}}+\alpha_{{3}}+\alpha_{{4}} \right) }},\\-{\frac {1}{\alpha_{{3}}
 \left( \alpha_{{3}}+\alpha_{{4}} \right)  \left( \alpha_{{2}}+\alpha_
{{3}}+\alpha_{{4}} \right)  \left( \alpha_{{1}}+\alpha_{{2}}+\alpha_{{
3}}+\alpha_{{4}} \right) }},-{\frac {1}{ \left( \alpha_{{2}}+\alpha_{{
3}} \right) \alpha_{{4}}\alpha_{{1}}\alpha_{{3}}}},\\{\frac {1}{\alpha_{
{4}}\alpha_{{3}} \left( \alpha_{{2}}+\alpha_{{3}} \right)  \left(
\alpha_{{1}}+\alpha_{{2}}+\alpha_{{3}} \right) }},{\frac {1}{\alpha_{{
4}} \left( \alpha_{{1}}+\alpha_{{2}}+\alpha_{{3}} \right) \alpha_{{1}}
\alpha_{{3}}}},{\frac {1}{\alpha_{{4}}\alpha_{{3}} \left( \alpha_{{2}}
+\alpha_{{3}} \right)  \left( \alpha_{{1}}+\alpha_{{2}}+\alpha_{{3}}
 \right) }},\\-{\frac {1}{\alpha_{{4}}\alpha_{{3}} \left( \alpha_{{2}}+
\alpha_{{3}} \right)  \left( \alpha_{{1}}+\alpha_{{2}}+\alpha_{{3}}
 \right) }},{\frac {1}{ \left( \alpha_{{2}}+\alpha_{{3}}+\alpha_{{4}}
 \right) \alpha_{{3}}\alpha_{{4}}\alpha_{{1}}}},\\-{\frac {1}{ \left(
\alpha_{{2}}+\alpha_{{3}}+\alpha_{{4}} \right) \alpha_{{3}}\alpha_{{4}
} \left( \alpha_{{1}}+\alpha_{{2}}+\alpha_{{3}}+\alpha_{{4}} \right) }
},-{\frac {1}{ \left( \alpha_{{1}}+\alpha_{{2}}+\alpha_{{3}}+\alpha_{{
4}} \right) \alpha_{{3}}\alpha_{{4}}\alpha_{{1}}}},\\-{\frac {1}{
 \left( \alpha_{{2}}+\alpha_{{3}}+\alpha_{{4}} \right) \alpha_{{3}}
\alpha_{{4}} \left( \alpha_{{1}}+\alpha_{{2}}+\alpha_{{3}}+\alpha_{{4}
} \right) }},{\frac {1}{ \left( \alpha_{{2}}+\alpha_{{3}}+\alpha_{{4}}
 \right) \alpha_{{3}}\alpha_{{4}} \left( \alpha_{{1}}+\alpha_{{2}}+
\alpha_{{3}}+\alpha_{{4}} \right) }},
\end{multline*}
respectively; $b$ maps
$$
(1,1,\omega_{{4}},\omega_{{3}},\omega_{{4}}),\; (\omega_{{4}},1,1,\omega_{{3}},\omega_{{4}}),
(\omega_{{4}},\omega_{{3}},1,1,\omega_{{4}}),\; (1,\omega_{{3}},\omega_{{4}},\omega_{{3}},1),\;
(\omega_{{4}},\omega_{{3}},\omega_{{4}},1,1)
$$
to
$$
{\frac {1}{\alpha_{{4}}\alpha_{{3}}}},\;-{\frac {1}{ \left( \alpha_{{3}}+\alpha_{{4}} \right) \alpha_{{4}}}},\;
-{\frac {1}{\alpha_{{3}} \left(
\alpha_{{3}}+\alpha_{{4}} \right) }},\;-{\frac {1}{\alpha_{{4}}\alpha_{{
3}}}},\;{\frac {1}{\alpha_{{4}}\alpha_{{3}}}},
$$
respectively; $c$ maps
$$
(1,1,1,\omega_{{1}},1),\;(\omega_{{2}},1,\omega_{{2}},\omega_{{1}},1),\;(1,\omega_{{1}},1,1,1),
\;(1,\omega_{{1}},\omega_{{2}},1,\omega_{{2}}),\;(\omega_{{2}},\omega_{{1}},\omega_{{2}},\omega_{{1}},\omega_{{2}})
$$
to
$$
{\frac {1}{\alpha_{{2}}\alpha_{{1}}}},\;-{\frac {1}{\alpha_{{2}} \left( \alpha_{{1}}+\alpha_{{2}} \right) }},\;
-{\frac {1}{ \left(
\alpha_{{1}}+\alpha_{{2}} \right) \alpha_{{1}}}},\;
-{\frac {1}{\alpha_{{
2}} \left( \alpha_{{1}}+\alpha_{{2}} \right) }},\;{\frac {1}{\alpha_{{2}
} \left( \alpha_{{1}}+\alpha_{{2}} \right) }},
$$
respectively. One can easily check that $a$ is a twisted product of $b$ and $c$ in the sense of~(\ref{eq:bt7:9}),
which in our case is
$$
c(\lm)=b(\lm_1,\lm_2,\lm_3,\lm_9,\lm_{10})(\lm_1\lm_2\lm_3c(\lm_4,\lm_5,\lm_6,\lm_7,\lm_8)).
$$

\def\bs{\\[-4pt]}

\end{document}